\documentclass[a4paper,11pt]{amsart}

\usepackage[utf8]{inputenc}
\usepackage{lmodern}
\usepackage{url}
\usepackage{amsmath, amsthm, amssymb, amscd, accents, bm}
\usepackage{mathtools}
\usepackage{mdwlist}
\usepackage{paralist}
\usepackage{graphicx}
\usepackage{datetime}
\usepackage{hyperref}
\usepackage{caption}
\usepackage[sort,nocompress]{cite}
\mathtoolsset{showonlyrefs}
\usepackage{units}
\usepackage{tikz}
\usepackage{colonequals}

\newtheorem{definition}{Definition}[section]
\newtheorem{theorem}[definition]{Theorem}
\newtheorem{lemma}[definition]{Lemma}
\newtheorem{corollary}[definition]{Corollary}
\newtheorem{proposition}[definition]{Proposition}

\newtheorem{remark}[definition]{Remark}
\newtheorem*{question}{Question}

\def\N{{\mathbb N}}
\def\Z{{\mathbb Z}}
\def\R{{\mathbb R}}
\def\T{{\mathbb T}}
\def\C{{\mathbb C}}

\newcommand{\Cd}{{\C^d}}
\newcommand{\lt}{{L^2(\R)}}

\newcommand{\re}{{\mathrm{Re}}}
\newcommand{\im}{{\mathrm{Im}}}

\newcommand{\ft}{{\mathcal{F}}}

\DeclareMathOperator{\dist}{dist}
\DeclareMathOperator{\argmin}{argmin}

\makeatletter
\newcommand\thankssymb[1]{\textsuperscript{\@fnsymbol{#1}}}

\allowdisplaybreaks

\begin{document}

\title[Phase retrieval and perturbation of Liouville sets]{Phase retrieval in Fock space and \\ perturbation of Liouville sets}

\author[Philipp Grohs]{Philipp Grohs}
\address{Faculty of Mathematics, University of Vienna, Oskar-Morgenstern-Platz 1, 1090 Vienna, Austria}
\address{Research Network DataScience@UniVie, University of Vienna, Kolingasse 14-16, 1090 Vienna, Austria}
\address{Johann Radon Institute of Applied and Computational Mathematics, Austrian Academy of Sciences, Altenbergstrasse 69, 4040 Linz, Austria}
\email{philipp.grohs@univie.ac.at}

\author[Lukas Liehr]{Lukas Liehr}
\address{Faculty of Mathematics, University of Vienna, Oskar-Morgenstern-Platz 1, 1090 Vienna, Austria}
\email{lukas.liehr@univie.ac.at}

\author[Martin Rathmair]{Martin Rathmair}
\address{Institut de Math\'ematiques de Bordeaux, Universit\'e Bordeaux, UMR CNRS 5251, 351 Cours
de la Lib\'eration 33405, Talence, France}
\email{martin.rathmair@math.u-bordeaux.fr}


\begin{abstract}
We study the determination of functions in Fock space from samples of their absolute value, known as the phase retrieval problem in Fock space. An important finding in this research field asserts that phaseless sampling on lattices of arbitrary density renders the problem unsolvable. The present study establishes solvability when using irregular sampling sets of the form $A \cup B \cup C$, where $A, B,$ and $C$ constitute perturbations of a Liouville set, i.e., a set with the property that all functions in Fock space bounded on the set are constant. The sets $A, B,$ and $C$ adhere to specific geometrical conditions of closeness and noncollinearity. We show that these conditions are sufficiently generic so as to allow the perturbations to be chosen also at random. By proving that Liouville sets occupy an intermediate position between sets of stable sampling and sets of uniqueness, we obtain the first construction of uniqueness sets for the phase retrieval problem in Fock space having a finite density. The established results apply to the Gabor phase retrieval problem in subspaces of $L^2(\mathbb{R})$, where we derive additional reductions of the size of uniqueness sets: for the class of real-valued functions, uniqueness is achieved from two perturbed lattices; for the class of even real-valued functions, a single perturbation suffices, resulting in a separated set.
\end{abstract}

\maketitle

\section{Introduction}

We are concerned with the problem of whether functions $F\in \mathcal{F}_\alpha(\mathbb{C})$ are determined from samples $\left(|F(u)|\right)_{u\in \mathcal{U}}$ of their absolute value on a set of sampling locations $\mathcal{U} \subseteq \mathbb{C}$.
Here,
$$
\mathcal{F}_\alpha(\mathbb{C}):=\left\{F:\mathbb{C}\to \mathbb{C}: \ F\textrm{ entire, }\int_{\mathbb{C}}|F(z)|^2 e^{-\alpha |z|^2} \,\mbox{d}A(z) < \infty \right\}
$$
denotes the Fock space with Euclidean area measure $\mbox{d}A(z)$ on $\C$. Since any function $H = \tau F$ with $\tau\in \mathbb{T} \coloneqq \{z\in \mathbb{C}: |z|=1 \}$ produces the same samples of its absolute values as $F$, we introduce the equivalence relation $$F\sim H \iff \exists \tau \in \T : H = \tau F,$$ and denote $\mathcal{U} \subseteq \mathbb{C}$ a \emph{uniqueness set for the phase retrieval problem in $\mathcal{F}_\alpha(\mathbb{C})$}, if the implication 
\begin{equation} \label{eq:uniqueness_set}
        \left(|F(u)|\right)_{u \in \mathcal{U}} = \left(|H(u)|\right)_{u \in \mathcal{U}} \implies F\sim H
\end{equation}
holds true for every $F,H \in \mathcal{F}_\alpha(\mathbb{C})$. The construction of uniqueness sets for the phase retrieval problem with desirable properties, as well as stable and efficient reconstruction algorithms, are of key importance in a number of applications of current interest, including diffraction imaging \cite{111,nature2}, audio processing \cite{nature3,pruuvsa2017phase,griffin1984signal}, and quantum mechanics \cite{appl1,Luef2019}. This is due to the fact that functions in Fock spaces arise naturally from the Gabor transform $\mathcal{G}f(x,\omega):= \int_\mathbb{R} f(t) e^{-\pi  (t-x)^2} e^{-2\pi i \omega t} \, \mbox{d}t$ of $f\in L^2(\mathbb{R})$, in the sense that the Bargmann transform  
\begin{equation}\label{eq:BargmannTransform}
   \mathcal{B}f(z) \coloneqq  \mathcal{G}f(\re{(z)},-\im{(z)}) \cdot e^{-\pi i \re{(z)}\im{(z)} +  \frac{\pi}{2} |z|^2}
\end{equation}
is a unitary operator from $\lt$ onto $\ft_\pi(\C)$. In addition, the Gabor transform plays a central role in signal processing and time-frequency analysis \cite{Groechenig}.  As a result, this problem has experienced a recent surge of research activity \cite{grohsliehr1,att1,att2,alaifari2020phase,AlaifariDaubechiesGrohsYin,alaifariGrohs,daubcahill,GrohsKoppensteinerRathmair,grohsLiehr3,grohsLiehr4,multiwindow}, with investigations conducted from various perspectives. These include finite-dimensional Gabor phase retrieval problems \cite{Bojarovska2016,Strohmer}, numerical reconstruction algorithms \cite{eldar,Iwen2022}, group theoretical settings \cite{fuhr2023phase,bartusel2023phase}, stability analysis \cite{frr, freeman2,freeman3,GrohsRathmair}, quantum harmonic analysis \cite{Luef2019}, and the study of the closely related Fourier phase retrieval problem \cite{bianchi}.

In order to be of practical value, a set $\mathcal{U}$ of sampling locations needs to satisfy the key property of having a finite density. A set $\mathcal{U} \subseteq \C$ is said to have finite density (or: is relatively separated), if
$$
\sup_{z \in \C} \#(\mathcal{U} \cap B_1(z)) < \infty,
$$
where $\#(\Omega)$ denotes the number of elements in a set $\Omega$, and $B_r(z) = \{ w \in \C : |w-z| \leq r \}$ denotes the closed ball of radius $r > 0$ around $z \in \C$. The Lebesgue measure of $B_r(z)$ is denoted by $|B_r(z)|$. Moreover, $\mathcal{U}$ is said to be uniformly distributed, if there exists a non-negative real number $D(\mathcal{U})$ satisfying
$$
\#(\mathcal{U} \cap B_r(z)) = D(\mathcal{U})|B_r(z)| + o(r^2), \quad r \to \infty,
$$
uniformly with respect to $z \in \C$. In this case, the quantity $D(\mathcal{U})$ is called the uniform density of $\mathcal{U}$. It corresponds to the average number of points in $\mathcal{U}$ per unit ball. Finally, we say that $\mathcal{U}$ is separated (or: uniformly discrete), if
$$
\delta(\mathcal{U}) \coloneqq \inf_{\substack{u,u' \in \mathcal{U} \\ u \neq u'}} |u - u'| > 0.
$$
The quantity $\delta(\mathcal{U})$ is called the separation constant of $\mathcal{U}$. Notice, that a set has finite density, if and only if it is a finite union of separated sets.

Compared to the phase retrieval problem, the uniqueness problem of $F\in \mathcal{F}_\alpha(\mathbb{C})$ from (non-phaseless) samples $(F(\lambda))_{u \in \mathcal{U}}$ is much better understood. Recall, that a set $\Lambda \subseteq \C$ is said to be a \emph{uniqueness set} for $\ft_\alpha(\C)$, if the only function in $\ft_\alpha(\C)$ that vanishes on $\Lambda$ is the zero function (we note, that this notion has to be distinguished from the notion of a \emph{uniqueness set for the phase retrieval problem in $\ft_\alpha(\C)$}). Uniqueness sets with finite density have already been constructed in the 1990s:
results by Lyubarskii \cite{lyubarskiiframes} and Seip-Wallstén \cite{Seip+1992+91+106,Seip1992} show that any $F\in \mathcal{F}_\alpha(\mathbb{C})$ is stably determined by its samples $(F(u))_{u\in \mathcal{U}}$, provided that $\mathcal{U}$ is separated, and its lower Beurling density exceeds the transition value $\frac{\alpha}{\pi}$, also known as the Nyquist rate.  Further studies related to uniqueness sets in Fock space can be found in \cite{ASCENSI2009277,BELOV2020438,aadi2022zero,WANG200460,enstad2022coherent}.
In most applications, it is desirable to sample along a regular set, which is why lattices constitute an important choice for sampling sets. Recall, that a lattice $\Lambda \subseteq \C$ with periods $\omega_1,\omega_2 \in \C \setminus \{ 0 \}$ is a set of the form
$$
\Lambda = \left \{ m\omega_1 + n \omega_2 : m,n \in \Z \right \},
$$
where $\im \, \frac{\omega_2}{\omega_1} > 0$. The quantity $s(\Lambda) > 0$ denotes the area of a period parallelogram of $\Lambda$, and is given by the formula
$$
s(\Lambda) = \im(\overline{\omega_1}\omega_2).
$$
The density of a lattice $\Lambda$ coincides with the reciprocal of $s(\Lambda)$. 
A characterization of uniqueness sets for $\mathcal{F}_\alpha(\mathbb{C})$ with a lattice structure was obtained by Perelomov \cite{perelomov71}. It states that any $F\in \mathcal{F}_\alpha(\mathbb{C})$ is uniquely determined by its samples $(F(\lambda))_{\lambda\in \Lambda}$, if and only if $s(\Lambda)\le \frac{\pi}{\alpha}$ \cite{perelomov71}.

The sampling problem from phaseless samples behaves in a completely different way. Recent work \cite{grohsLiehr3,grohsLiehr4,alaifari2020phase} has shown, that no lattice can be a uniqueness set for the phase retrieval problem in $\mathcal{F}_\alpha(\mathbb{C})$, irrespective of how small $s(\Lambda)$ (or equivalently, how large $D(\Lambda)$) is. A natural question is then, whether the phaseless sampling problem is simply unsolvable or, alternatively, if the lack of uniqueness is due to algebraic obstructions caused by the regular lattice structure. In other words:
\begin{question}
    Do there exist uniqueness sets of finite density for the phase retrieval problem?     
\end{question} 
The main contribution of this article is to provide a positive answer to this question, thereby affirming that the lack of uniqueness is indeed due to the regular structure of lattice sampling sets:

\vspace{0.2cm}

\begin{quote}
{\it 
    While there are no lattices $\Lambda$ which are uniqueness sets for the phase retrieval problem in $\mathcal{F}_\alpha(\mathbb{C})$, for every $d > \frac{12 \alpha}{\pi}$ there exist uniqueness sets $\mathcal{U}$ for the phase retrieval problem in $\mathcal{F}_\alpha(\mathbb{C})$ having finite uniform density $D(\mathcal{U}) = d$. These uniqueness sets can be constructed as a union of three perturbations of a lattice, where the perturbations can be either deterministic or random. Consequently, solving the phase retrieval problem from samples with finite density is possible, but it requires irregular sampling.
    }
\end{quote}

\vspace{0.2cm}

The latter implies that -- contrary to most other sampling problems -- for the phase retrieval problem, the main obstruction to being a uniqueness set is not due to insufficient density, but due to the regular structure inherent in lattices. 
In the context of completeness of systems of discrete translates, and universal sampling of band-limited functions, similar phenomena have been encountered in articles by Olevskii, Ulanovskii and Lev \cite{olevskii:completeness, olevskii:universal1, olevskii:universal2, olevskii:discretetranslates, lev2024completenessuniformlydiscretetranslates}.

In the subsequent sections, we present a notably more comprehensive approach to constructing uniqueness sets. The method builds upon the perturbation of Liouville sets.

\begin{definition}
Let $\Lambda \subseteq \C$, and let $V$ be a set of entire functions. We say that $\Lambda$ is a Liouville set for $V$, if every $f \in V$ that is bounded on $\Lambda$ is a constant function.
\end{definition}

In the extremal case, we have that $\Lambda=\C$ is a Liouville set for the collection of all entire functions, according to the classical Liouville theorem. We will be predominantly interested in sets $\Lambda$ which are of discrete nature. 
If $\Lambda=(\lambda_n)_{n\in\N}$ satisfies $|\lambda_n|\to \infty$ as $n\to \infty$, Weierstrass factorization theorem guarantees the existence of a nonzero entire function which vanishes on $\Lambda$. In particular, such a set $\Lambda$ cannot be a Liouville set for the space of entire functions, which demonstrates that the restriction to a subclass $V$ is then inevitable.

From a more general perspective, a Liouville set $\Lambda\subsetneq \C$ can be understood as a stronger version of the classical Liouville principle: to conclude that a given function is constant it suffices to check that it is bounded on $\Lambda$.  In the context of harmonic functions on the graph $\Z^2$ a strong Liouville principle has been established in \cite{buhovsky:discreteharmonicfct}.
Furthermore, it is natural to consider the problem of identifying Liouville sets that are in a sense minimal. We address this question for $V=\ft_\alpha(\C)$.

Recall that a discrete set $\Gamma \subseteq \C$ is said to be a set of stable sampling for $\ft_\alpha(\C)$, if there exist constants $A,B>0$ such that
$$
A\| F \|_{\ft_\alpha(\C)}^2 \leq \sum_{\gamma \in \Gamma} e^{-\alpha|\gamma|^2}|F(\gamma)|^2 \leq B \| F \|_{\ft_\alpha(\C)}^2
$$
for all $F \in \ft_\alpha(\C)$.
Sets of stable sampling were characterized by Lyubarskii, Seip, and Wallstén in \cite{Seip1992,Seip+1992+91+106,lyubarskiiframes}. The results of the present manuscript show that Liouville sets occupy an intermediate position between sets of stable sampling and uniqueness sets. Moreover, they exhibit a close association with sets of uniqueness for the phase retrieval problem. This can be summarized as follows:

\vspace{0.2cm}
\begin{quote}
{\it 
    Every set of stable sampling for $\ft_\alpha(\C)$ is a Liouville set for $\ft_\alpha(\C)$, and every Liouville set for $\ft_\alpha(\C)$ is a uniqueness set for $\ft_\alpha(\C)$. 
    Uniqueness sets for the phase retrieval problem in $\ft_\alpha(\C)$ can be obtained by combining three perturbations of a Liouville set for $\ft_{4\alpha}(\C)$. 
    The perturbations may either be selected according to a deterministic condition or at random.
}

\end{quote}

\vspace{0.2cm}

Our findings have a series of implications for the Gabor phase retrieval problem. To fix notation, we say that a set $\mathcal{U} \subseteq \R^2$ is a \emph{uniqueness set for the Gabor phase retrieval problem in} $X \subseteq \lt$, if
\begin{equation}\label{eq:def_uniqueness_gabor}
        \ \left(|\mathcal{G}f(u)|\right)_{u \in \mathcal{U}} = \left(|\mathcal G h (u)|\right)_{u \in \mathcal{U}} \implies f\sim h
\end{equation}
holds true for every $f,h \in X$. In a similar fashion as above, $f \sim h$ indicates the existence of a value $\tau \in \T$ such that $f=\tau h$. 
Note that if $X$ consists of real-valued functions exclusively, then one can even conclude that $f=\pm h$, provided that the respective Gabor transforms coincide on a uniqueness set $\mathcal{U}$ for Gabor phase retrieval in $X$.
Denoting by $L^2(\R,\R)$ the space of all real-valued functions in $\lt$, and by $L^2_e(\R,\R)$ the space of all even functions in $L^2(\R,\R)$, we can summarize our findings in the following way:

\vspace{0.2cm}

\begin{quote}
{\it 
    For every $d > 12$ there exists a uniqueness set $\mathcal{U}$ for the Gabor phase retrieval problem in $\lt$ having uniform density $D(\mathcal{U}) = d$. Moreover, for every $d > 6$ (resp. $d > 3$), there exist uniqueness sets for the Gabor phase retrieval problem in $L^2(\R,\R)$ (resp. $L^2_e(\R,\R)$) of uniform density $D(\mathcal{U})=d$. In all scenarios, these uniqueness sets are formed by combining perturbations of a subset of a lattice and can be either deterministic or random. Notably, the uniqueness set for the space $L^2_e(\R,\R)$ can be created by a single perturbation of a lattice, resulting in a separated set.
    }
\end{quote}

\vspace{0.2cm}

A prominent application of the Gabor phase retrieval problem occurs in Ptychography, a powerful diffraction microscopy technique capable of achieving significantly higher resolution than methods based on conventional optics \cite{rodenburg2019ptychography,111}. Beyond its mathematical contributions, our work for the first time provides rigorous results on how to collect measurements from diffraction patterns such that unique recoverability can be guaranteed.

\section{Results}

\subsection{Closeness and noncollinearity}\label{sec:closeness}

In order to state the general version of our findings, we require two geometrical concepts related to closeness and noncollinearity.

\begin{definition}\label{def:f_close}
Let $A = (a_\lambda)_{\lambda \in \Lambda} \subseteq \C, \,B = (b_\lambda)_{\lambda \in \Lambda} \subseteq \C$ be two sequences indexed by $\Lambda$, and let $f : \C \to [0,\infty)$ be a function. We say that $A$ is $f$-close to $B$ if there exists a constant $\kappa>0$ such that
$$
|a_\lambda - b_\lambda| \leq \kappa f(b_\lambda), \quad \lambda \in \Lambda.
$$
Moreover, $A$ is said to be uniformly close to $B$, if $A$ is $f$-close to $B$ with respect to a constant function.
\end{definition}

Notice, that uniform closeness is the common notion of closeness in the sampling theory and Fock space literature \cite{Seip1992,zhu:fock}. Definition \ref{def:f_close} can be understood as a more flexible notion of closeness as it allows that the permitted deviation depends on the position.
For instance, if $f : \C \to [0,\infty)$ is such that $f(z) \to 0$ as $|z| \to \infty$,
then -- provided that $\Lambda$ is unbounded -- the notion of $f$-closeness is strictly stronger than uniform closeness, as the margin of the allowed amount of perturbation becomes arbitrarily small when $|\lambda| \to \infty$. In addition to a closeness concept, we require geometrical notions related to noncollinearity.

\begin{definition}
    Let $a,b,c \in \C$, and let $\varphi_1,\varphi_2,\varphi_3 \in (0,\pi)$ be the interior angles of the triangle with vertices $a,b,c$, enumerated such that $\varphi_1 \leq \varphi_2\leq \varphi_3$. Then we define
    $$
    \varphi(a,b,c) \coloneqq \varphi_2 \in [0,\tfrac{\pi}{2}]. \footnote{
    By convention, if the triangle is degenerated (that is, if $a,b,c$ are collinear) we define 
    $\varphi(a,b,c)\coloneqq0$.}
    $$
\end{definition}

The previous definition allows to introduce noncollinear resp. uniformly noncollinear sequences.

\begin{definition}
Let $J$ be an index set.
We say that $(a_j)_{j \in J}, (b_j)_{j \in J}, (c_j)_{j \in J} \subseteq \C$ are noncollinear sequences if for every $j \in J$ it holds that $a_j,b_j,c_j$ are noncollinear. We say that $(a_j)_{j \in J}, (b_j)_{j \in J}, (c_j)_{j \in J}$ are uniformly noncollinear if there exists an angle $\theta \in (0,\frac{\pi}{2})$ such that
$$
\theta \leq \varphi(a_j,b_j,c_j), \quad j \in J.
$$
\end{definition}

\subsection{Deterministic perturbations of Liouville sets}

Having settled the required notions of closeness and noncollinearity, we are prepared to state the first main result of the article, which proves uniqueness of the phase retrieval problem in Fock space from unions of three perturbed Liouville sets.

\begin{theorem}\label{thm:Main1}
Let $\alpha > 0$, and let $\Lambda \subseteq \C$ be a Liouville set for $\ft_{4\alpha}(\C)$. Further, let $f: \C \to [0,\infty)$ be given by $f(z) = e^{-\gamma |z|^2}$ with $\gamma > 2\alpha$.
Suppose that $A=(a_\lambda)_{\lambda\in\Lambda},B=(b_\lambda)_{\lambda\in\Lambda},C=(c_\lambda)_{\lambda\in\Lambda} \subseteq \C$ are $f$-close to $\Lambda$.
If
\begin{equation}\label{cond:anglesgeneral}
   \exists \beta\in (0,\gamma-2\alpha), \, \exists L>0:\quad \frac{|\lambda| e^{-\beta|\lambda|^2}}{\varphi(a_\lambda,b_\lambda,c_\lambda)}  \le L,\quad \lambda\in\Lambda,
\end{equation}
then
$$
\mathcal{U} \coloneqq A \cup B \cup C 
$$
is a uniqueness set for the phase retrieval problem in $\mathcal{F}_\alpha(\C)$. In particular, condition (\ref{cond:anglesgeneral}) holds if $A,B,C$ are uniformly noncollinear.
\end{theorem}

We emphasize that Theorem \ref{thm:Main1} crucially depends on the interaction of two conditions: $f$-closeness of $A, B, C$ to a Liouville set $\Lambda$ with respect to the function $f(z)=e^{-\gamma |z|^2}$, and the angle condition \eqref{cond:anglesgeneral}. 
While the combination of these two conditions is sufficient for $\mathcal{U}$ to be a uniqueness set, dropping either of them generally does not guarantee unique recoverability. For details on this matter, we direct the reader to Remark \ref{rem:sharpness}.

\subsection{Sets of stable sampling, lattices, and Hardy's uncertainty principle}\label{sec:2.2}

The previous result may appear somewhat abstract, as it involves obtaining the uniqueness set for phase retrieval through perturbations of Liouville sets. To provide more clarity,  
we explore the concept of Liouville sets and its relationship with commonly studied sets of stable sampling and sets of uniqueness. Our main result in this direction reads as follows.

\begin{theorem}\label{thm:Main2}
    Let $\mathcal{S}_S$ be the collection of all subsets of $\C$ that contain a set of stable sampling for $\ft_\alpha(\C)$, let $\mathcal{S}_L$ be the collection of all Liouville sets for $\ft_\alpha(\C)$, and let $\mathcal{S}_U$ be the collection of all uniqueness sets for $\ft_\alpha(\C)$. Then it holds that
    $$
    \mathcal{S}_S \subsetneq \mathcal{S}_L \subsetneq \mathcal{S}_U.
    $$
\end{theorem}

In addition to the inclusion given in Theorem \ref{thm:Main2}, we provide a characterization of all Liouville sets possessing a lattice structure. This characterization is expressed in terms of the lattice density, as expected when drawing parallels with related statements on sets of stable sampling and sets of uniqueness.

\begin{theorem}\label{thm:Main3}
    A lattice $\Lambda \subseteq \C$ is a Liouville set for $\ft_\alpha(\C)$ if and only if $s(\Lambda) < \frac{\pi}{\alpha}$.
\end{theorem}

It follows from Theorem \ref{thm:Main3}, that the assertion made in Theorem \ref{thm:Main1} remains valid for any lattice $\Lambda \subseteq \C$ that satisfies $s(\Lambda) < \frac{\pi}{4\alpha}$. In this scenario, the sets $A$, $B$, and $C$ in Theorem \ref{thm:Main1} represent perturbations of the lattice $\Lambda$. It is evident that, in this case, the collective set $A \cup B \cup C$ exhibits finite uniform density, yielding the first example in current literature of a uniqueness set for phase retrieval in Fock space with this property. This effectively addresses the inquiry posed in the introduction.

Having settled the relation to sets of stable sampling and uniqueness sets, we further notice, that each Liouville set for $\ft_\pi(\C)$ gives rise to a Gabor transform related Hardy uncertainty principle. The classical form of Hardy's uncertainty principle states that if a function $f \in \lt$ satisfies
$$
|f(x)| \leq c_1 e^{-a\pi x^2}, \quad |\hat f(\omega)| \leq c_2 e^{-b\pi \omega^2},
$$
for some $c_1,c_2,a,b > 0$ with $ab=1$, then $f(x) = c_3 e^{-a\pi x^2}$ for some $c_3 \in \C$ \cite{10.1112/jlms/s1-8.3.227,Folland1997}. In the above inequality, $\hat f$ denotes the Fourier transform of $f$, defined on $L^1(\R) \cap \lt$ by $\hat f(\omega) = \int_\R f(x)e^{-2\pi i \omega x} \, \mbox{d}x$, and extends to $\lt$ in the usual way. By replacing the Fourier transform with the Gabor transform (or more generally, windowed Fourier transforms), one arrives at a time-frequency analytic version of Hardy's uncertainty principle, originally established by Gröchenig and Zimmermann \cite{gr_zimmermann_2001}. The results presented in the present section can be seen as discretized versions of Hardy's uncertainty principle for the Gabor transform. The forthcoming statement elaborates on this interpretation.

\begin{theorem}\label{cor:Hardy}
    Let $\Lambda \subseteq \C \simeq \R^2$ be a set of stable sampling for $\ft_\pi(\C)$, and let $f \in \lt$. If there exists a constant $C \geq 0$ such that
    \begin{equation}\label{eq:hardy2}
    |\mathcal{G}f(\lambda)| \leq C e^{-\frac{\pi}{2}|\lambda|^2},\quad \lambda \in \Lambda,
\end{equation}
then $f(x)=ce^{-\pi x^2}$ for some $c \in\C$.
\end{theorem}

\begin{figure}
\centering
\hspace*{-0.4cm}
  \includegraphics[width=12.7cm]{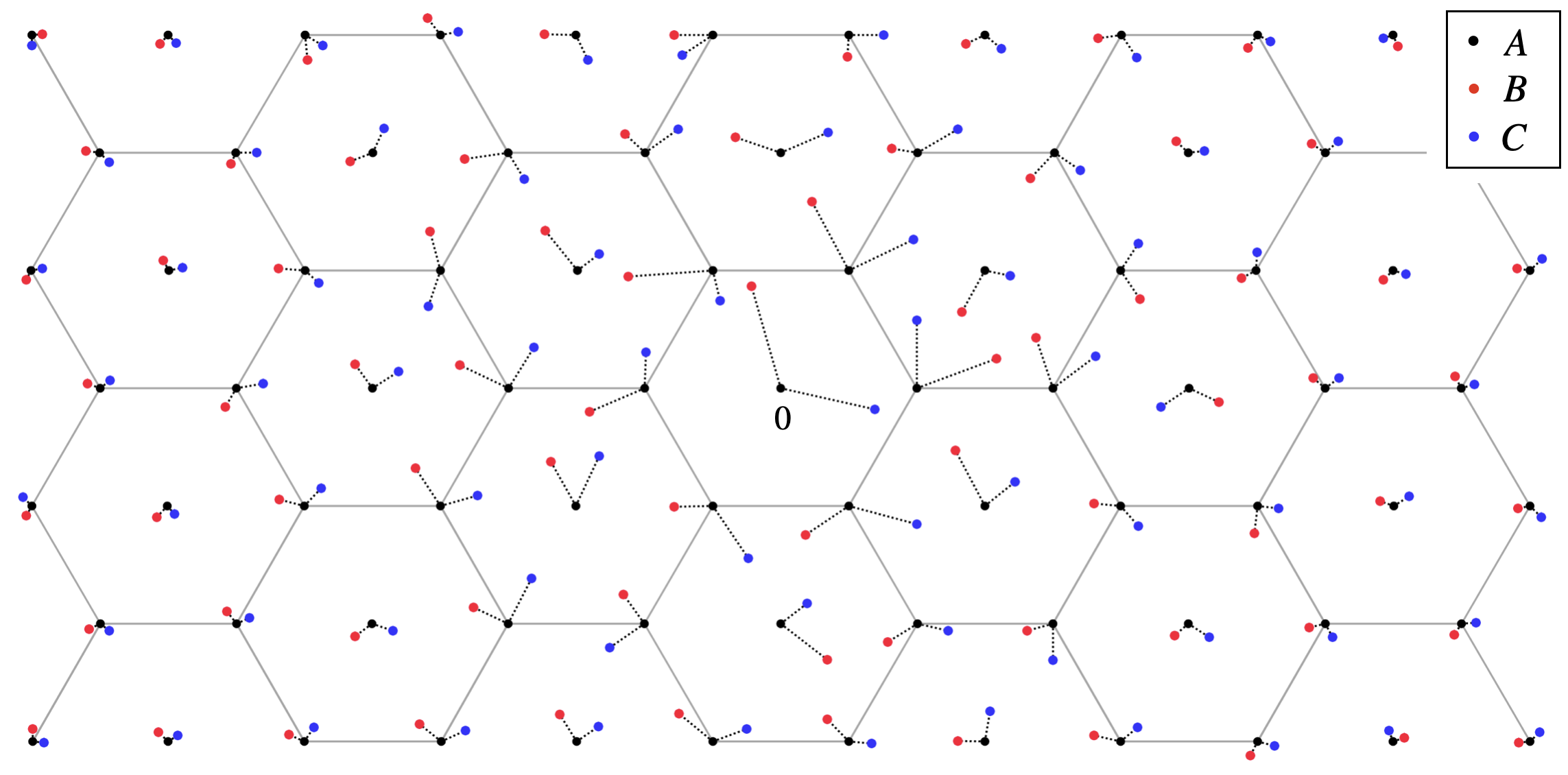}
\caption{Visualization of Theorem \ref{thm:Main1}: the sets $A,B$, and $C$ are uniformly noncollinear and $f$-close to a sufficiently dense (hexagonal) lattice $\Lambda$, which is a Liouville set for $\ft_\alpha(\C)$. For simplicity, we consider $A$ to be the lattice itself, i.e., $A = \Lambda$. Theorem \ref{thm:Main1} establishes the unique determination (up to a global phase) of every function in the Fock space by its absolute values located on the union $A \cup B \cup C$. However, it is important to note that sampling only on the lattice $A=\Lambda$ does not guarantee uniqueness.
}
\label{fig:1}
\end{figure}

\subsection{Random perturbation of Liouville sets}

We now turn our attention towards investigating whether probabilistic uniqueness statements are attainable. The concept of contemplating random perturbations of a Liouville set is rooted in the fact that one could contend that the condition \eqref{cond:anglesgeneral} outlined in Theorem \ref{thm:Main1} is rather generic, and it is possible that random perturbations might satisfy this condition almost surely. Indeed, the following can be said.

\begin{theorem}\label{thm:Main4}
Let $\alpha>0$, let $f:\C\to [0,\infty)$ be given by $f(z) = e^{-\gamma|z|^2}$, $\gamma>2\alpha$, and let $\Lambda\subseteq \C$ be a  Liouville set for $\ft_{4\alpha}(\C)$ of finite density.
If 
$$
Z_{\lambda,\ell}, \quad (\lambda,\ell)\in \Lambda \times \{1,2,3\}
$$
is a sequence of independent complex random variables, where each $Z_{\lambda,\ell}$ is uniformly distributed on the disk $B_{f(\lambda)}(\lambda)$, then $\mathcal{U}=\{Z_{\lambda,\ell} : (\lambda,\ell)\in \Lambda \times \{1,2,3\} \}$ is a uniqueness set for the phase retrieval problem in $\ft_\alpha(\C)$ almost surely.
\end{theorem}

\begin{figure}
\centering
\hspace*{-0.4cm}
  \includegraphics[width=10cm]{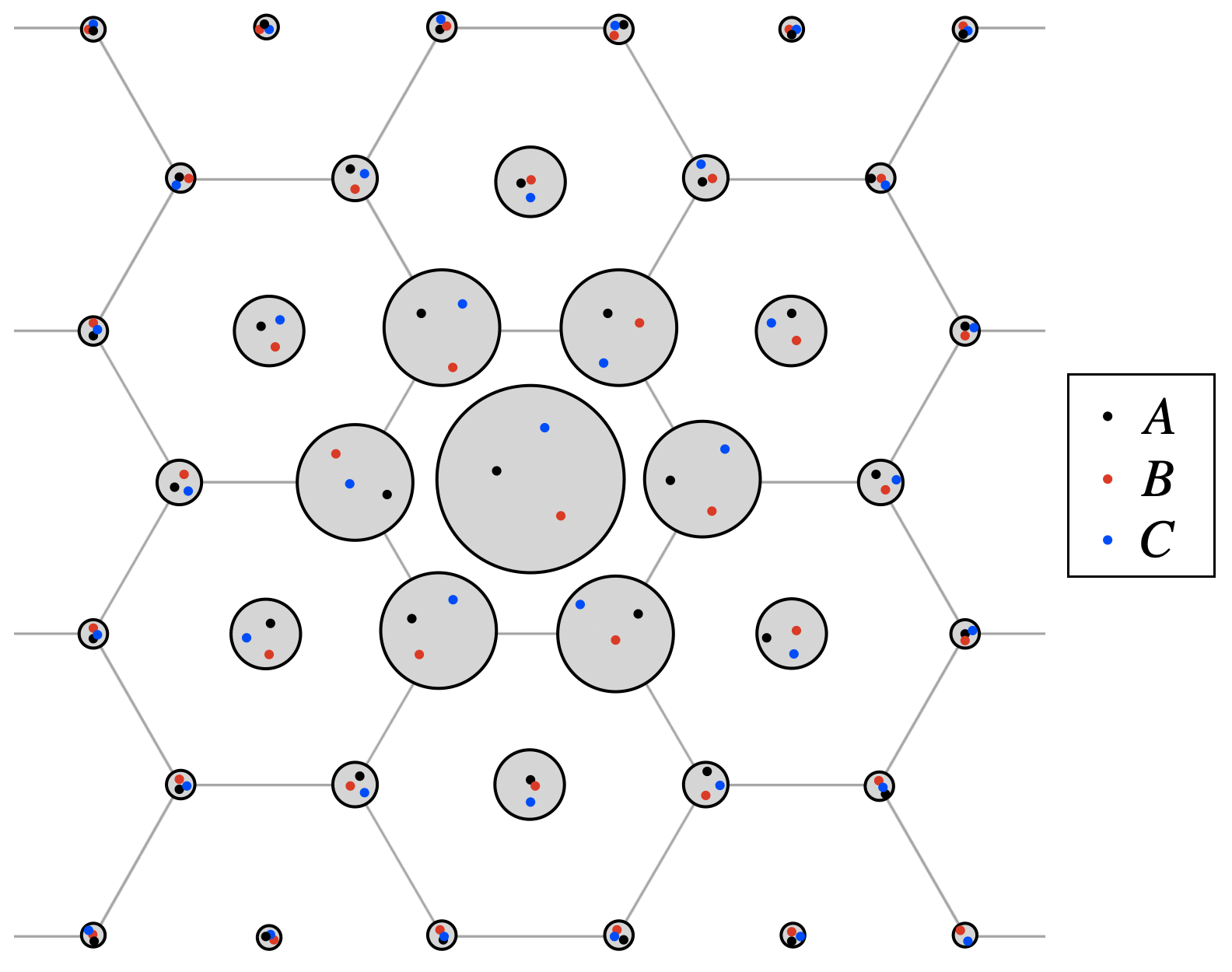}
\caption{Visualization of Theorem \ref{thm:Main4}: the grey disks represent neighborhoods around points $\lambda$ of a sufficiently dense (hexagonal) lattice $\Lambda$. From each disk, three points are selected at random, resulting in three sets $A,B$, and $C$. With probability 1, the union $A \cup B \cup C$ forms a uniqueness set for the phase retrieval problem in Fock space.}
\label{fig:2}
\end{figure}

\subsection{Gabor phase retrieval}\label{sec:resultsrandomgabor}

In the remainder of this section, we identify $\R^2$ with the complex plane $\C$. All definitions introduced earlier, naturally transfer to $\R^2$ without further explanations. Throughout the present section, the function $f:\R^2\to [0,\infty)$ is defined by $f(z) = e^{-\gamma|z|^2}$ with $\gamma>2\pi$. We recall that $\mathcal{U} \subseteq \R^2$ is considered a uniqueness set for the Gabor phase retrieval problem in $X \subseteq \lt$, if the implication \eqref{eq:def_uniqueness_gabor} is satisfied. We begin with an immediate consequence of Theorem \ref{thm:Main4}.

\begin{theorem}\label{thm:Main5}
Let $\Lambda \subseteq \R^2$ be a lattice satisfying $D(\Lambda) > 4$. Suppose that
$$
Z_{\lambda,\ell}, \quad (\lambda,\ell)\in \Lambda \times \{1,2,3\}
$$
is a sequence of independent ($\, \R^2$-valued) random variables, where each $Z_{\lambda,\ell}$ is uniformly distributed on the disk $B_{f(\lambda)}(\lambda)$. 
Then $\mathcal{U}=\{Z_{\lambda,\ell} : (\lambda,\ell)\in \Lambda \times \{1,2,3\}\}$ 
is a uniqueness set for the Gabor phase retrieval problem in $L^2(\R)$ almost surely.
\end{theorem}

We proceed with the introduction of a class of special lattices and define
$$
\mathcal{L}:= \left\{\begin{pmatrix}
    p &0\\
    0 &q
\end{pmatrix}
\Z^2: \,p,q\in \R\setminus\{0\}
\right\}
\cup 
\left\{\begin{pmatrix}
    p &p\\
    q &-q
\end{pmatrix}
\Z^2: \,p,q\in \R\setminus\{0\}
\right\}.
$$
Such lattices are precisely the ones which are invariant with respect to reflection across either of the two coordinate axes. The class $\mathcal{L}$ contains the important class of separable lattices. It is important to emphasize, that lattices that belong to $\mathcal{L}$ do not form uniqueness sets for the Gabor phase retrieval problem in $L^2(\R,\R)$, the space of real-valued functions in $\lt$. That is, for each $\Lambda \in \mathcal{L}$, there exist two functions $f,h \in L^2(\R,\R)$, satisfying the two conditions
$$
|\mathcal{G}f(z)| = |\mathcal{G}h(z)|, \quad z \in \Lambda, \quad \mathrm{and} \quad  f \not \sim h,
$$
see the proof of \cite[Theorem 3.13]{grohsLiehr3}. Clearly, since $L^2(\R,\R) \subseteq \lt$, Theorem \ref{thm:Main1} is applicable and establishes uniqueness in $L^2(\R,\R)$ through the use of three perturbed lattices. However, the space $L^2(\R,\R)$ is significantly smaller in comparison to $\lt$, leading one to anticipate that a reduced sampling set would suffice for uniqueness. Indeed, the following theorem asserts that the union of two random perturbations of a lattice in $\mathcal{L}$ is adequate for achieving uniqueness.

\begin{theorem}\label{thm:Main6}
Suppose that $\Lambda \in \mathcal{L}$ satisfies $D(\Lambda)>4$. Let
$$
Z_{\lambda,\ell}, \quad (\lambda,\ell)\in \Lambda \times \{1,2\}
$$
be a sequence of independent ($\, \R^2$-valued) random variables, where each $Z_{\lambda,\ell}$ is uniformly distributed on the disk $B_{f(\lambda)}(\lambda)$. Then $\mathcal{U}=\{Z_{\lambda,\ell} : (\lambda,\ell)\in \Lambda \times \{1,2\}\}$ is a uniqueness set for Gabor phase retrieval in $L^2(\R,\R)$ almost surely. 
\end{theorem}

If we further restrict the problem to the space $L^2_e(\R,\R)$ of even real-valued functions, then a single perturbation suffices, and the uniqueness set becomes separated.

\begin{theorem}\label{thm:Main7}
Suppose that $\Lambda \in \mathcal{L}$ satisfies $D(\Lambda) > 4$. Let
$$
Z_{\lambda}, \quad \lambda \in \Lambda
$$
be a sequence of independent ($\, \R^2$-valued) random variables, where each $Z_{\lambda}$ is uniformly distributed on the disk $B_{f(\lambda)}(\lambda)$. Then $\mathcal{U}=\{Z_{\lambda} : \lambda\in \Lambda\}$ is separated and a uniqueness set for Gabor phase retrieval in $L_e^2(\R,\R)$ almost surely. 
\end{theorem}

A further objective is to identify the
range of densities (depending on the function class $X\subseteq \lt$ considered) such that uniqueness sets for the given class exist.
In the randomized setting of Theorem \ref{thm:Main6} and Theorem \ref{thm:Main7}, we end up with uniqueness sets of density $> 8$ and $> 4$, respectively. The densities of the respective sets can be arbitrarily close to the lower bounds.
As it turns out, the density can be further pushed down by employing a more elaborate strategy to select the sampling points. Specifically, we have the following result.

\begin{theorem}\label{thm:density_lower_bound}
Let $j \in \{ 1,2,3 \}$. For every $d > d_j$, there exists a uniformly distributed uniqueness set $\mathcal{U}_j$ for the Gabor phase retrieval problem in $X_j$ having uniform density $D (\mathcal{U}_j) = d$, provided that
\begin{enumerate}
    \item $d_1=12$ and $X_1 = \lt$,
    \item $d_2 = 6$ and $X_2 = L^2(\R,\R)$,
    \item $d_3 = 3$ and $X_3 = L^2_e(\R,\R)$.
\end{enumerate}
For $X_3 = L^2_e(\R,\R)$, the uniqueness set $\mathcal{U}_3$ can be chosen to be separated.
\end{theorem}

\begin{figure}
\centering
\hspace*{-0.4cm}
  \includegraphics[width=12.7cm]{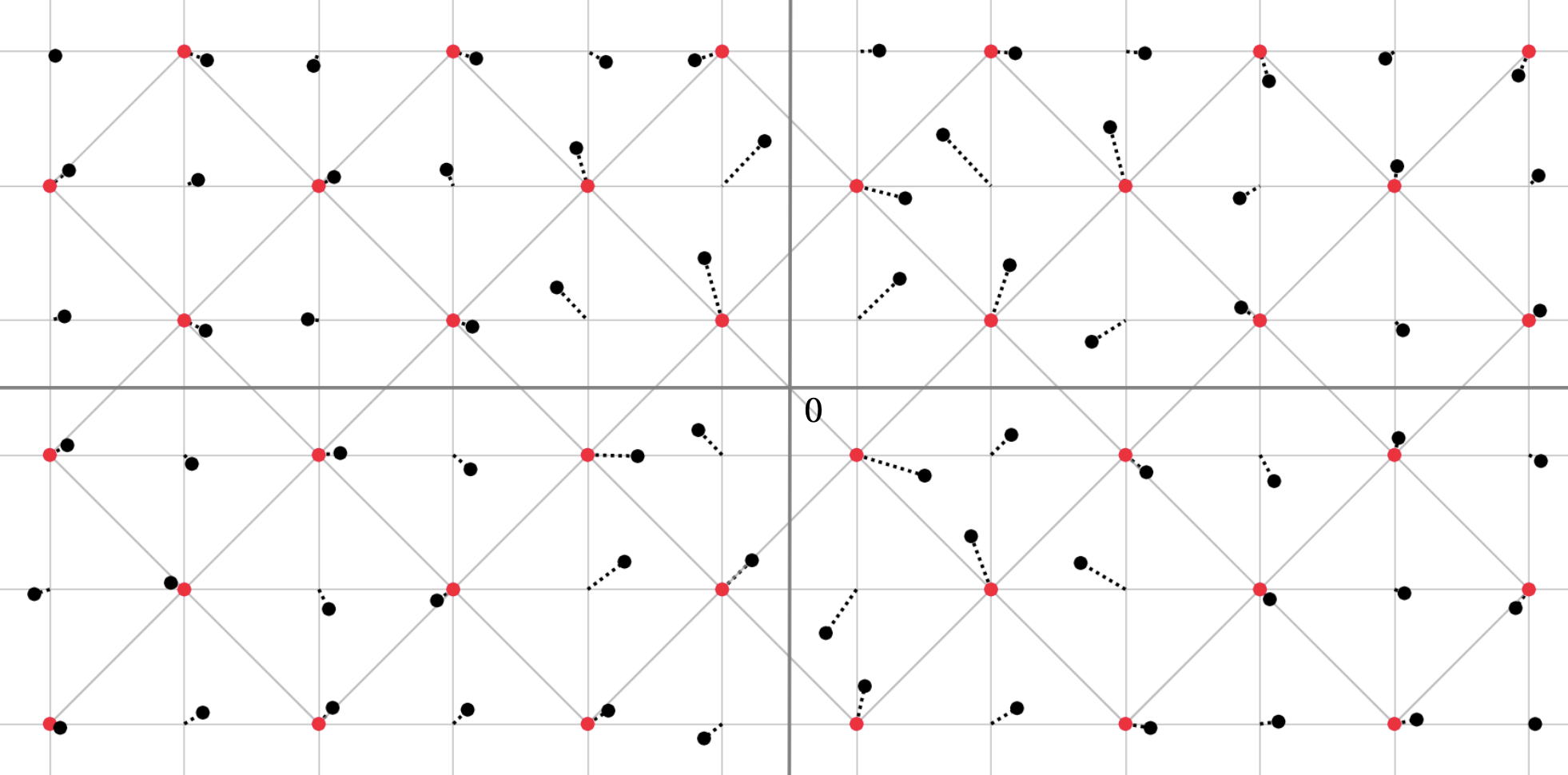}
\caption{This figure depicts an example related to Theorem \ref{thm:density_lower_bound}(2). The vertical and horizontal lines induce a (shifted) square lattice $\Lambda$ of density $>4$. The black points are perturbations of $\Lambda$. The red points depict a sublattice $\Lambda' \subseteq \Lambda$ of density $D(\Lambda') = \frac{1}{2}D(\Lambda)$. The union of the black and red points forms a uniqueness set for the Gabor phase retrieval problem in $L^2(\R,\R)$, and has density $> 6$. The density can be as close to $6$ as we please.
}
\label{fig:3}
\end{figure}

\begin{figure}
\centering
\hspace*{-0.4cm}
  \includegraphics[width=12.7cm]{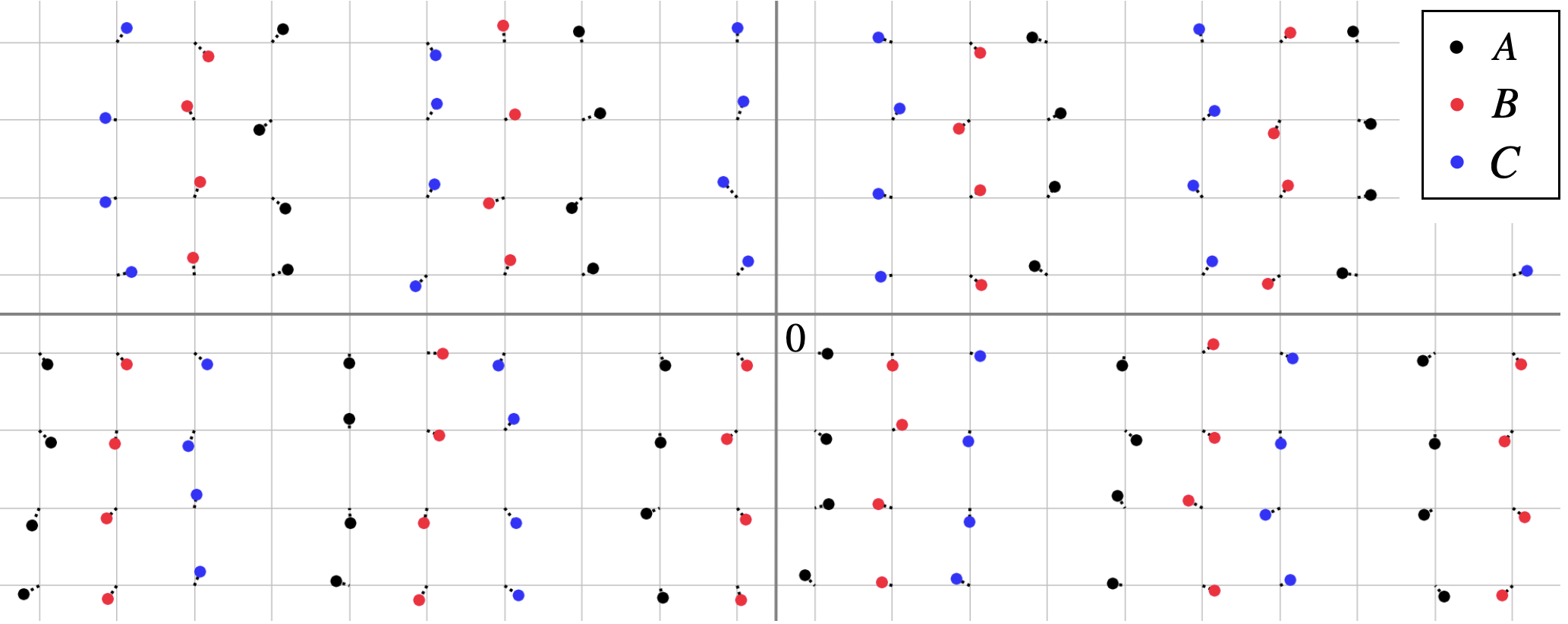}
\caption{This figure depicts an example related to Theorem \ref{thm:density_lower_bound}(3). The gray mesh induces a square lattice $\Lambda$ of density $>4$. The points $A,B,C$ are perturbations of points of $\Lambda$, with the property that their union, $A \cup B \cup C$, forms a uniqueness set for the Gabor phase retrieval problem in $L^2_e(\R,\R)$, having density $>3$. This density can be as close to $3$ as we please. In addition, $A \cup B \cup C$ is separated.
}
\label{fig:4}
\end{figure}

\subsection{Previous work}

\subsubsection{Square-root sampling}

To our knowledge, the first construction of a discrete uniqueness set for the phase retrieval problem in Fock spaces is provided in the recent work \cite{squaresampling}. This uniqueness set is of the form $\Lambda_{\sqrt{\cdot}} = \left\{\pm \sqrt{n} \pm i \sqrt{m}:\ n,m\in \mathbb{N}_0 \right\}$. However, since the set $\Lambda_{\sqrt{\cdot}}$ does not have finite density and is not uniformly distributed, this result is of limited practical value.

\subsubsection{Multi-window approach}

In another line of work, the authors show that if for each $F\in \mathcal{F}_\pi(\C)$ one additionally obtains phaseless samples of 3 judiciously chosen functions $F_1,F_2,F_3$, one can even reconstruct uniquely from phaseless samples on any lattice of density $\ge 4$ \cite{multiwindow}. 
The reformulation of the result in terms of time-frequency analysis amounts to employing four suitably chosen window functions. 
This corresponds to $4$-fold sampling on a lattice of density $\ge 4$.

\subsubsection{Restriction to subspaces}

Lastly, we note that phaseless sampling on lattices in $\mathcal{F}_\pi(\C)$ can be done uniquely by restricting the problem to specific proper subspaces of $\mathcal{F}_\pi(\C)$. Notable examples of such subspaces include the image of $L^2(I)$, where $I \subseteq \R$ is compact, under the Bargmann transform or the corresponding image of certain shift-invariant spaces \cite{grohsliehr1,grohs2022completeness,att2,WELLERSHOFF2024127692}.

\section{Notation}

The following notation is used throughout the remainder of the article. Given two vectors $z,w\in\C^d$, we use the abbreviation $z\cdot w= \sum_{k=1}^d z_iw_i$.
The Euclidean length of $z$ will be denoted by $|z|$.
The symbols $\mathbb{D}, \mathbb{H}_+, \mathbb{H}_-$ denote the open unit disk, the open upper half plane, and the open lower half plane, respectively.
By $\partial \Omega$ we denote the boundary of a set $\Omega \subseteq \C$. Further, we use the notation $\overline{\Omega}=\{\overline{z}: z\in\Omega\}$ and $-\Omega=\{-z: z\in\Omega\}$.
Given three points $a,b,c\in \C$ in the complex plane, we denote by 
$\Delta(a,b,c)$ the triangle with vertices $a,b,c$. The distance of $\Omega$ to a point $z \in \C$ is defined by $\mathrm{dist}(\Omega,z) \coloneqq \inf \{ |w-z| : w \in \Omega \}$.

By $\mathcal{O}(\C^d)$ we denote the collection of entire functions of $d$ complex variables.
For $F\in\mathcal{O}(\C^d)$ we use the notation $F^*(z)=\overline{F(\bar{z})}$. Note that $F^*\in \mathcal{O}(\C^d)$.
The Lebesgue measure on the $d$-dimensional complex space $\C^d$ will be denoted by $A(z)$.
For $\alpha>0$, we introduce a family of probability measures by defining
$$
\mbox{d}\mu_\alpha(z) = \left(\frac{\alpha}{\pi} \right)^d e^{-\alpha|z|^2} \,\mbox{d}A(z).
$$
The Fock space $\ft_\alpha(\C^d)$ is defined by
$$
\ft_\alpha(\C^d) = \left\{F\in\mathcal{O}(\C^d):\, \|F\|_{L^2(d\mu_\alpha)}<\infty \right\}.
$$
Equipped with the inner product
$$
\langle F,G\rangle_{\ft_\alpha(\C^d)} = \int_{\C^d} F(z)\overline{G(z)}\,\mbox{d}\mu_\alpha(z),
$$
the space $\ft_\alpha(\C^d)$ forms a reproducing kernel Hilbert space, and the reproducing vectors are given by
$$
k_z(w) = e^{\alpha w\cdot \bar{z}}.
$$
Notice, that every $F \in \ft_\alpha(\C^d)$ satisfies the pointwise growth estimate
\begin{equation}\label{eq:growth_fock_space}
    |F(z)| \leq \| F \|_{\ft_\alpha(\C^d)} e^{\frac{\alpha}{2}|z|^2}, \quad z \in \C^d.
\end{equation}
The estimate \eqref{eq:growth_fock_space} in combination with the definition of the norm $\| \cdot \|_{\ft_\alpha(\C^d)}$, implies that for every $s \in \C^d$ and every $\varepsilon>0$ one has
\begin{equation}\label{eq:shift_episolon}
    F(\cdot + s) \in \ft_{\alpha + \varepsilon}(\C^d).
\end{equation}
For a detailed exposition of the theory of Fock spaces we refer to \cite{zhu:fock}.

\section{Liouville sets}\label{sec:liouville_sets}

This section is devoted to an investigation of Liouville sets for the Fock space. Recall that $\Lambda \subseteq \C$ is said to be a Liouville set for $V \subseteq \mathcal{O}(\C)$ if every $F \in V$ that is bounded on $\Lambda$ is a constant function.

\subsection{ Uniform closeness and Lagrange-type interpolation}

As introduced in Section \ref{sec:closeness}, a set $A$ is said to be uniformly close to $B$, if $A$ is $f$-close to $B$ with respect to a constant function $f(z) = \Delta \geq 0$. Within the context of this section, we need to address the concept of uniform closeness more precisely: we say that a set $A$ is $\Delta$-uniformly close to a set $B$, if there exists an index set $J$, and an enumeration of $A$ and $B$ with index set $J$, i.e., $A=(a_j)_{j \in J}, \, B=(b_j)_{j \in J}$, such that
$$
|a_j - b_j| \leq \Delta, \quad j \in J.
$$

For $\beta > 0$, we define the square lattice $\Lambda_\beta$ by
$$
\Lambda_\beta \coloneqq \sqrt{\frac{\pi}{\beta}}(\Z + i \Z) = (\lambda_{mn})_{m,n \in \Z}, \quad \lambda_{mn} \coloneqq \sqrt{\frac{\pi}{\beta}}(m+in).
$$

Following Seip and Wallstén \cite{Seip1992}, we define for a separated sequence $\Gamma = (\gamma_{mn})_{m,n \in \Z} \subseteq \C$ that is uniformly close to $\Lambda_\beta$ for some $\beta > 0$, the entire function $g_\Gamma : \C \to \C$ by
$$
    g_\Gamma(z) \coloneqq (z-\gamma_{00}) \cdot \prod_{\substack{(m,n) \in \Z^2 \\ (m,n) \neq (0,0) }} \left ( 1-\frac{z}{\gamma_{mn}} \right ) \exp \left ( \frac{z}{\gamma_{mn}} + \frac{1}{2} \frac{z^2}{\lambda_{mn}^2} \right ),
$$
where 
$$
\gamma_{00}=\underset{\substack{\gamma\in\Gamma\\ |\gamma|=\dist(\Gamma,0)}}{\argmin} \arg(\gamma)
$$
is the point which has minimal argument among all minimal length points in $\Gamma$. With this, $\Gamma\mapsto g_\Gamma$ is a well-defined map. The derivative of $g_\Gamma$ at points in $\Gamma$ satisfies the following lower bound \cite[Lemma 2.2]{Seip1992}.

\begin{lemma}\label{g_property}
    Let $\Gamma$ be separated and $\Delta$-uniformly close to $\Lambda_\beta$. Then there exist constants $C,c>0$ depending only on $\Delta$ and $\delta(\Gamma)$ such that
    $$
    |g'_\Gamma(\gamma)| \geq Ce^{-c|\gamma|\log |\gamma|} e^{\frac{\beta}{2}|\gamma|^2}, \quad \gamma \in \Gamma.
    $$
\end{lemma}

The introduction of $g_\Gamma$ leads to the following Lagrange-type interpolation formula for functions in Fock space \cite[Lemma 3.1]{Seip1992}.

\begin{lemma}
    Let $0 < \alpha < \beta$, and suppose that $\Gamma \subseteq \C$ is separated and uniformly close to $\Lambda_\beta$. Then every $F \in \ft_\alpha(\C)$\footnote{In fact, the interpolation formula is valid for a larger class of functions, namely for the space $\mathcal{O}(\C) \cap L^\infty(\C,d\mu_\alpha)$.} satisfies the interpolation formula
    $$
    F(z) = \sum_{\gamma \in \Gamma} \frac{F(\gamma) g_\Gamma(z)}{g'_\Gamma(\gamma)(z-\gamma)},
    $$
    where the series on the right-hand side converges absolutely and uniformly on compact subsets of $\C$.
\end{lemma}

As a last ingredient, we require the following elementary lemma on uniform closeness of shifts of sets.

\begin{lemma}\label{uc_Delta}
    Let $\Gamma$ be uniformly close to $\Lambda_\beta$. Then there exists a constant $\Delta \geq 0$ such that $\Gamma-s$ is $\Delta$-uniformly close to $\Lambda_\beta$ for every $s \in \C$.
\end{lemma}
\begin{proof}
Let $s \in \C$, and suppose that $\Delta' \geq 0$ is a constant such that $\Gamma$ is $\Delta'$-uniformly close to $\Lambda_\beta$. This clearly implies that $\Gamma-s$ is $\Delta'$-uniformly close to $\Lambda_\beta - s$. Since $\Lambda_\beta - s$ is $\sqrt{\frac{\pi}{2\beta}}$-uniformly close to $\Lambda_\beta$, it follows from triangle inequality, that $\Gamma-s$ is $\Delta$-uniformly close to $\Lambda_\beta$ with
    $$
    \Delta = \Delta' + \sqrt{\frac{\pi}{2\beta}}.
    $$
    Since $s$ was arbitrary, the statement is proved.
\end{proof}

\subsection{Liouville sets in Fock space}\label{subsec:LiouvilleFock}

The main results of the present section are based on the following key theorem.

\begin{theorem}\label{thm:uc_implies_liouville}
    Let $0 < \alpha < \beta$. If $\Gamma \subseteq \C$ is separated and uniformly close to $\Lambda_\beta$, then $\Gamma$ is a Liouville set for $\ft_\alpha(\C)$.
\end{theorem}
\begin{proof}
    Let $F \in \ft_\alpha(\C)$ and let $L>0$ such that
    \begin{equation}\label{L}
        |F(\gamma)| \leq L, \quad \gamma \in \Gamma.
    \end{equation}
    Let $\Delta \geq 0$ be given as in Lemma \ref{uc_Delta}, such that $\Gamma-s$ is $\Delta$-uniformly close to $\Lambda_\beta$ for every $s \in \C$. Since $F_s \coloneqq F(\cdot + s) \in \ft_{\alpha + \varepsilon}(\C)$ for every $s \in \C$ and every $\varepsilon>0$, we can expand every $F_s$ in terms of the Lagrange expansion with respect to the set $\Gamma-s$, resulting in the identity
    $$
    F_s(z) = \sum_{\gamma \in \Gamma - s} \frac{F_s(\gamma) g_{\Gamma-s}(z)}{g'_{\Gamma-s}(\gamma)(z-\gamma)}.
    $$
    The definition of $g_{\Gamma-s}$ implies that for every $s \in \C$, one has
    $$
    |g_{\Gamma-s}(0)| = \mathrm{dist}(\Gamma-s,0).
    $$
    Now let $s \in \C \setminus \Gamma$. This assumption is equivalent to the condition that $0 \not \in \Gamma-s$. For such $s$, we have
    \begin{equation}\label{eq:distance}
        \frac{|g_{\Gamma-s}(0)|}{|\gamma|} \leq \frac{|g_{\Gamma-s}(0)|}{\mathrm{dist}(\Gamma-s,0)} = 1, \quad \gamma \in \Gamma-s.
    \end{equation}
    Moreover, it holds that
    \begin{equation}\label{eq:L_bound}
        |F_s(\gamma)| \leq L, \quad \gamma \in \Gamma-s.
    \end{equation}
    Evaluating the Lagrange expansion of $F_s$ at zero for some $s \in \C \setminus \Gamma$, and using \eqref{eq:distance} and \eqref{eq:L_bound}, gives
    $$
    |F_s(0)| \leq L \sum_{\gamma \in \Gamma - s} \frac{1}{|g'_{\Gamma-s}(\gamma)|}.
    $$
    Since $\Gamma-s$ is $\Delta$-uniformly close to $\Lambda_\beta$ for every $s \in \C$, and since $\delta(\Gamma-s) = \delta(\Gamma)$ for every $s \in \C$, it follows from Lemma \ref{g_property}, that there exist constants $C,c>0$, independent of $s$, such that
    $$
    |F_s(0)| \leq \frac{L}{C} \sum_{\gamma \in \Gamma - s} e^{c|\gamma|\log |\gamma|} e^{-\frac{\beta}{2}|\gamma|^2}.
    $$
    We seek to bound the sum on the right hand side by a constant which is independent from $s\in\C$.
    To do so, note that there exists a constant $B>0$  such that 
    $$
    c r \log r - \frac\beta2 r^2 \le B - r, \quad r\ge 0.
    $$
    Moreover, since $\Gamma-s$ is $\Delta$-uniformly close to $\Lambda_\beta$, we can index $\Gamma-s$ in terms of $\Lambda_\beta$, $\Gamma-s = (\gamma_\lambda)_{\lambda\in\Lambda_\beta}$, and obtain that
    $$
    |\gamma_\lambda| \ge |\lambda|-|\gamma_\lambda-\lambda| \ge |\lambda|-\Delta,\quad \lambda\in\Lambda_\beta.
    $$
    With this, we obtain that 
    \begin{equation}
        \sum_{\gamma \in \Gamma - s} e^{c|\gamma|\log |\gamma|} e^{-\frac{\beta}{2}|\gamma|^2}
         \le \sum_{\lambda\in\Lambda_\beta} e^{B-|\gamma_\lambda|}
        \le  \sum_{\lambda\in\Lambda_\beta} e^{B+\Delta-|\lambda|}=:K <\infty,
    \end{equation}
    and $K$ does not depend on $s$ as desired.
    Consequently,
    $$
    |F_s(0)| = |F(s)| \leq \frac{LK}{C}, \quad s \in \C \setminus \Gamma.
    $$
    It follows from continuity, that $F$ is globally bounded. Liouville's theorem implies that $F$ is a constant function.
\end{proof}

We can deduce from the previous statement, that every subset of $\C$ that contains a set of stable sampling is a Liouville set. 

\begin{corollary}\label{cor:stable_sampling}
If $\Lambda \subseteq \C$ contains a set of stable sampling for $\ft_\alpha(\C)$, then $\Lambda$ is a Liouville set for $\ft_\alpha(\C)$.
\end{corollary}
\begin{proof}
    Suppose that $\Gamma \subseteq \Lambda$ is a set of stable sampling for $\ft_\alpha(\C)$. Then $\Gamma$ contains a separated subsequence $\Gamma'$ of lower Beurling density strictly greater than $\frac{\alpha}{\pi}$ \cite[Theorem 4.36]{zhu:fock}. This in turn implies that $\Gamma'$ contains a separated subsequence $\Gamma''$ that is uniformly close to $\Lambda_\beta$ for some $\beta > \alpha$ \cite[Lemma 4.31]{zhu:fock}. According to Theorem \ref{thm:uc_implies_liouville}, $\Gamma''$ is a Liouville set for $\ft_\alpha(\C)$. Since $\Gamma'' \subseteq \Lambda$, it follows that $\Lambda$ is a Liouville set for $\ft_\alpha(\C)$.
\end{proof}

Next, we turn towards a characterization of lattice Liouville sets in $\ft_\alpha(\C)$. To do so, we apply Perelomov's theorem \cite[Assertion 1]{perelomov71}.

\begin{theorem}[Perelomov]\label{thm:perelomov}
    A lattice $\Lambda \subseteq \C$ is a uniqueness set for $\ft_\alpha(\C)$ if and only if $s(\Lambda) \leq \frac{\pi}{\alpha}$. If $s(\Lambda) = \frac{\pi}{\alpha}$, then $\Lambda$ stays a uniqueness set for $\ft_\alpha(\C)$ on the removal of a single point, but fails to be a uniqueness set if more than one point is removed.
\end{theorem}

\begin{proof}[Proof of Theorem \ref{thm:Main3}]
Let $\Lambda \subseteq \C$ be a lattice. If $s(\Lambda) < \frac{\pi}{\alpha}$, then $\Lambda$ is a set of stable sampling for $\ft_\alpha(\C)$ \cite[Theorem 1.1]{Seip1992}. In view of Corollary \ref{cor:stable_sampling}, $\Lambda$ is a Liouville set for $\ft_\alpha(\C)$.

It remains to show, that if a lattice $\Lambda$ is a Liouville set for $\ft_\alpha(\C)$, then $s(\Lambda) < \frac{\pi}{\alpha}$. 
To this end, we observe that every Louville set for $\ft_\alpha(\C)$ is a uniqueness set for $\ft_\alpha(\C)$. According to Theorem \ref{thm:perelomov}, it must hold that $s(\Lambda) \leq \frac{\pi}{\alpha}$. It remains to show, that any lattice at the critical rate $s(\Lambda) = \frac{\pi}{\alpha}$ cannot be a Liouville set for $\ft_\alpha(\C)$. Assume the contrary, i.e., $\Lambda$ is a Liouville set for $\ft_\alpha(\C)$ satisfying $s(\Lambda) = \frac{\pi}{\alpha}$. Since every Liouville set is a uniqueness set, and, in addition, remains a Liouville set on the removal of any finite number of points, we end up with a contradiction to Theorem \ref{thm:perelomov}.
\end{proof}

\begin{remark}[Critical density]\label{rem:critical_density}
    According to Theorem \ref{thm:Main3}, a lattice $\Lambda$ with $s(\Lambda) = \frac{\pi}{\alpha}$ is not a Liouville set for $\ft_\alpha(\C)$. Hence, there exists a function in $\ft_\alpha(\C)$, that is not a constant function, and, in addition, is bounded on $\Lambda$. Such a function can be constructed by means of a modified Weierstrass $\sigma$-function. Recall, that the Weierstrass $\sigma$-function associated to a lattice $\Lambda$ with periods $\omega_1,\omega_2$ is given by
$$
\sigma(z) = z \prod_{\lambda \in \Lambda \setminus \{ 0\}} \left ( 1-\frac{z}{\lambda} \right ) \exp \left ( \frac{z}{\lambda} + \frac{z^2}{2\lambda^2} \right ).
$$
The function $\sigma$ is quasi-periodic, i.e, there exist values $\eta_1,\eta_2 \in \C$ such that
\begin{equation}
    \begin{split}
        \sigma(z+\omega_1) &= - \sigma(z) e^{\eta_1 z + \frac{1}{2}\eta_1 \omega_1}, \\
        \sigma(z+\omega_2) &= - \sigma(z) e^{\eta_2 z + \frac{1}{2}\eta_2 \omega_2}.
    \end{split}
\end{equation}
For $\omega_1,\omega_2,\eta_1,\eta_2$ given as above, define $a(\Lambda) \in \C$ via 
$$
    a(\Lambda) = \frac{1}{2} \frac{\eta_2 \overline{\omega_1} - \eta_1 \overline{\omega_2}}{\omega_1 \overline{\omega_2} - \omega_2 \overline{\omega_1}},
$$
and set $\sigma_\Lambda(z) \coloneqq \sigma(z)e^{a(\Lambda) z^2}$. We call $\sigma_\Lambda$ the modified Weierstrass $\sigma$-function associated to $\Lambda$. The modified Weierstrass $\sigma$-function was introduced by Hayman in \cite{hayman_1974}, and played an important role in a series of articles by Gröchenig and Lyubarskii on sets of stable sampling in Fock space \cite{Groechenig2009,GROCHENIG2,GROCHENIG3}. The term $e^{a(\Lambda) z^2}$ in the definition of $\sigma_\Lambda$ serves as a correction term and gives rise to the growth estimate
\begin{equation}\label{sigma_upper_bound}
        |\sigma_\Lambda(z)| \le C e^{\frac\pi{2 s(\Lambda)}|z|^2}, \quad z\in\C,
\end{equation}
with $C>0$ a constant depending only on $\Lambda$ \cite[Proposition 3.5]{Groechenig2009}.
    Now consider two distinct lattice points $\lambda,\lambda' \in \Lambda$, and define
    $$
    Q(z) \coloneqq \frac{\sigma_\Lambda(z)}{(z-\lambda)(z-\lambda')}.
    $$
Since $\lambda$ and $\lambda'$ are zeros of $\sigma_\Lambda$ it follows that $Q$ is holomorphic. Using \eqref{sigma_upper_bound} and the identity $s(\Lambda) = \frac{\pi}{\alpha}$, it follows that
\begin{equation}
    \begin{split}
        & \int_\C |Q(z)|^2 e^{-\alpha |z|^2} \, \mbox{d}A(z) \\
        & \leq \int_{B_{2|\lambda-\lambda'|}(\lambda)} |Q(z)|^2 e^{-\alpha |z|^2} \,\mbox{d}A(z) + C^2 \int_{\C \setminus B_{2|\lambda-\lambda'|}(\lambda)} \frac{1}{|(z-\lambda)(z-\lambda')|^2} \,\mbox{d}A(z).
    \end{split}
\end{equation}
The first integral on the right-hand side of the latter inequality exists since the integrand is continuous. The decay of the function $z \mapsto \frac{1}{(z-\lambda)(z-\lambda')}$ on $\C \setminus B_{2|\lambda-\lambda'|}(\lambda)$ implies that the second integral exists as well. Consequently, $Q \in \ft_\alpha(\C)$.  As $Q$ vanishes on $\Lambda\setminus \{\lambda,\lambda'\}$, we have that $Q$ is necessarily bounded on $\Lambda$, while not being a constant.
\end{remark}

We proceed by establishing Theorem \ref{thm:Main2}.

\begin{proof}[Proof of Theorem \ref{thm:Main2}]
    The statements discussed in the present section imply that $\mathcal{S}_S \subseteq \mathcal{S}_L \subseteq \mathcal{S}_U$.
    
    To show that the inclusions are proper, we first note that every lattice $\Lambda \subseteq \C$ satisfying $s(\Lambda)=\frac{\pi}{\alpha}$ is a uniqueness set for $\ft_\alpha(\C)$ but not a Liouville set for $\ft_\alpha(\C)$. This shows that $\mathcal{S}_L \subsetneq \mathcal{S}_U$.

    Finally, we construct a Liouville set for $\ft_\alpha(\C)$ that does not contain a set of stable sampling for $\ft_\alpha(\C)$. To do so, let $L_1,L_2,L_3\subseteq \C$ be three distinct lines in the plane intersecting at $0$. Then the connected components of $\C\setminus( L_1\cup L_2\cup L_3)$ are six sectors. 
    Let us further assume that the three lines are chosen in such a way that each of these sectors has an acute angle at the origin.
    If $F \in \ft_\alpha(\C)$ is bounded on $L_1 \cup L_2 \cup L_3$, then the estimate $|F(z)| \leq \| F \|_{\ft_\alpha(\C)} e^{\frac{\alpha}{2}|z|^2}$ in conjunction with the Phragmén–Lindelöf principle \cite[Theorem 10, p. 68]{Young} shows that $F$ is bounded in each of the six sectors. Consequently, $L_1 \cup L_2 \cup L_3$ is a Liouville set for $\ft_\alpha(\C)$.
    On the other hand, the complement of $L_1 \cup L_2 \cup L_3$ contains arbitrarily large balls, implying that the lower Beurling density of every subset of $L_1 \cup L_2 \cup L_3$ is zero. Since every set of stable sampling contains a subset of positive lower Beurling density \cite[Theorem 1.1]{seipBulletin}, it follows that $L_1 \cup L_2 \cup L_3$ does not contain a set of stable sampling.
\end{proof}

We conclude this section with the proof of Theorem \ref{cor:Hardy}, which states that every sampling set induces a Hardy uncertainty principle.

\begin{proof}[Proof of Theorem \ref{cor:Hardy}]
    Let $\Lambda \subseteq \C$ be a set of stable sampling for $\ft_\pi(\C)$, and let $f \in \lt$ such that
\begin{equation}
    |\mathcal{G}f(\lambda)| \leq C e^{-\frac{\pi}{2}|\lambda|^2},\quad \lambda \in \Lambda.
\end{equation}
The latter bound is equivalent to the condition that the Bargmann transform $\mathcal{B}f$ is bounded on $\bar \Lambda$.
Note that as $\Lambda$ is a set of stable sampling, so is $\bar \Lambda$.
As per Theorem \ref{thm:Main2} we have that $\bar \Lambda$ is a Liouville set for $\ft_\pi(\C)$. Hence, there exists a constant $c \in \C$ with $\mathcal B f = c$. Since $\mathcal{B}^{-1}(1)(x) = e^{-\pi x^2}$, we conclude that $f(x)=ce^{-\pi x^2}$.
\end{proof}

\subsection{Comparison to previous work}

Given a constant $\tau > 0$, we define
$$
V_\tau \coloneqq \left \{ F \in \mathcal{O}(\C) : \limsup_{r \to \infty} \frac{\log M_F(r)}{r^2} < \frac{\pi}{2\tau} \right \},
$$
where $M_F(r) = \max_{\theta \in \R} |F(re^{i\theta})|$ denotes the maximum modulus function of $F$. For the extremal case $\tau = \infty$, we define $V_\infty$ as the collection of all entire functions $F \in \mathcal{O}(\C)$ of order two and minimal type, that is, 
$$
\limsup_{r \to \infty} \frac{\log M_F(r)}{r^2} = 0.
$$
In his exposition on interpolatory function theory, Whittaker demonstrated that the lattice $\Z + i \Z$ is a Liouville set for $V_\infty$ \cite{whittaker}. Iyer proved a generalization, namely that $\Z + i \Z$ is a Liouville set for for $V_1$ \cite{iyer1936}. The same statement was independently proven a short time later by Pfluger \cite{pfluger}. Maitland showed that if $\Lambda = (\lambda_{nm})_{n,m \in \mathbb{Z}} \subseteq \mathbb{C}$ is a separated sequence satisfying
$$n-1 \leq  \re(\lambda_{nm}) \leq n, \quad m-1 \leq  \im(\lambda_{nm}) \leq m, \quad n,m \in \Z,$$
then $\Lambda$ is a Liouville set for $V_1$ \cite{Maitland1939OnAF}. It follows directly from the estimate \eqref{eq:growth_fock_space}, that the sets obtained by Whittaker, Iyer, Pfluger, and Maitland are Liouville sets for $\ft_\alpha(\C)$, provided that $0 < \alpha < \pi$. However, these results are restrictive in the sense that they neither cover lattices that lack a square structure nor sets that are uniformly close to a square lattice. In contrast, Theorem \ref{thm:uc_implies_liouville} above covers the results of Whittaker, Iyer, Pfluger, and Maitland, it implies that all sets of stable sampling are Liouville sets, and it yields a full characterization of lattice Liouville sets in terms of the Nyquist rate. The Nyquist rate is the decisive quantity that also characterizes all lattices that are sets of stable sampling or uniqueness sets.

Lastly, we reference Cartwright's discrete Liouville theorem pertaining to functions in $V_1$, which exhibit bounded behavior both along a sector's boundary and at lattice points within the same sector \cite{cartwright}.

\section{Auxiliary results}
In this section, we collect a couple of intermediate results which will be of fundamental importance for the proofs of the main results of this article.

\subsection{A first uniqueness statement}
We begin by establishing a result which will serve as a key piece to bridge the gap from the discrete sampling set to the full complex plane.

\begin{proposition}\label{lma:uniqueness_lemma}
Let $\Lambda \subseteq \C$ be a uniqueness set for $\ft_\alpha(\C)$. Further, let $F,H \in \ft_\alpha(\C)$, and let
$$
G \coloneqq FH(FH'-F'H).
$$
If $|F(\lambda)| = |H(\lambda)|$ for all $\lambda \in \Lambda$, and in addition, $G=0$, then $F \sim H$.
\end{proposition}
\begin{proof}
Since $G$ is the product of the three entire functions $F, H$, and $FH'-F'H$, it follows that at least one of them must be the zero function. If $F=0$ then the property that $|F(\lambda)| = |H(\lambda)|$ for all $\lambda \in \Lambda$ implies that $H$ vanishes on $\Lambda$. Since $\Lambda$ is a uniqueness set for $\ft_\alpha(\C)$ we have $H=0$. In particular $F \sim H$. An analogous argument shows that if $H=0$ then $F \sim H$. Finally, suppose that $FH'-F'H=0$ and that $F$ does not vanish identically. If $\Omega \subseteq \C$ is a zero-free region of $F$, then for every $z \in \Omega$ we have
$$
\left ( \frac{H(z)}{F(z)} \right )' = 0.
$$
Thus, there exists  $\tau \in \C$ such that $F=\tau H$. Since $F \neq 0$, it follows from the assumption on $\Lambda$ being a uniqueness set for $\ft_\alpha(\C)$, that there exists a $\lambda' \in \Lambda$ such that $F(\lambda') \neq 0$. Moreover, $|F(\lambda')| = |H(\lambda')|$. Consequently, $|F(\lambda')|(1-|\tau|) = 0$ which gives $|\tau|=1$.
\end{proof}

\subsection{First order phaseless information}

In the proof of the main result (Theorem \ref{thm:Main1}), the function $F'\overline{F}$, where $F \in \ft_\alpha(\C)$ is to be determined given its phaseless samples, plays a subtle but nevertheless central role.
We begin with making the observation that this information is encoded in $\nabla |F|^2$, the gradient of the squared modulus:
let $z_0\in\C$ be any point and let us identify $\C$ and $\R^2$ by virtue of $z=x+iy$.
Suppose we are given first order phaseless information
$(\partial_x |F|^2)(z_0)$ and $(\partial_y |F|^2)(z_0)$.
Since $F$ is holomorphic and since $\partial_z=\frac12 (\partial_x - i \partial_y)$, one can infer 
$$
(F'\overline{F})(z_0) = (\partial_z |F|^2)(z_0) = \frac12 (\partial_x |F|^2)(z_0) - \frac{i}2 (\partial_y |F|^2)(z_0),
$$
from the given information. The purpose of this section is to establish a more general and quantitative version of this consideration.
To this end, we denote for a function $Q:\C\simeq\R^2\rightarrow \C$, differentiable in the real variable sense (but not necessarily holomorphic), and $\theta\in [0,2\pi)$ an angle, its derivative at $z \in \C$ into $e^{i\theta}$-direction by
$$
\delta_\theta [Q](z)= \frac{\mbox{d}}{\mbox{d}t} Q(z+te^{i\theta})\Big|_{t=0}.
$$
We consider the case $Q=|F|^2-|H|^2$ with $F,H\in \ft_\alpha(\C)$ a pair of functions in the Fock space. 
If $\theta_1,\theta_2$ are distinct angles and if $\delta_{\theta_1}[Q], \delta_{\theta_2}[Q]$ vanish in the vicinity of a point $z_0\in\C$, we can control the deviation 
$$
|F'\overline{F} - H' \overline{H}|
$$
at $z_0$. This is the content of the following result.

\begin{proposition}\label{prop:zeroperturbation}
Let $\alpha,\varepsilon>0$, let $z_0\in\C$, and suppose that 
$$
0 \leq \eta\le  \min\left\{(2\alpha+\varepsilon)^{-1/2}, (2\alpha+\varepsilon)^{-1}|z_0|^{-1} \right\}.
$$
Moreover, let $F,H\in \ft_\alpha(\C)$, let  $\theta_1, \theta_2\in \R$ be two angles such that $\theta_1-\theta_2 \notin \pi \Z$, and assume that there exist $p_1,p_2 \in B_\eta(z_0)$ such that 
$$
\delta_{\theta_j}[|F|^2-|H|^2](p_j) = 0,\quad j\in\{1,2\}.
$$
Then it holds that 
$$
|(F'\overline{F}-H'\overline{H})(z_0)|\le M \frac{(|z_0|+1) e^{(\alpha+\frac\varepsilon2)|z_0|^2}}{|\sin(\theta_1-\theta_2)|} \eta,
$$
where $M$ is a constant depending only on $\alpha,\varepsilon,\|F\|_{\ft_\alpha(\C)}$, and $\|H\|_{\ft_\alpha(\C)}$.
\end{proposition}

Before we turn towards the proof of the latter statement we collect three intermediate results. These results involve the function
$$
\dist_\alpha(z,w) \coloneqq \|k_z- k_w\|_{\ft_\alpha(\C^d)}, \quad z,w\in\C^d,
$$
where $k_z$ denotes the reproducing vector in $\ft_\alpha(\C^d)$ with respect to $z \in \C^d$. The latter defines a metric on $\C^d$. The following result makes matters more explicit and provides a link to the Euclidean metric.

\begin{lemma}\label{lem:estdist}
    For all $z,w\in\C^d$ it holds that 
    $$
    \dist_\alpha(z,w) = \sqrt{e^{\alpha|z|^2}- 2 \re (e^{\alpha z\cdot\bar{w}}) + e^{\alpha|w|^2}}.
    $$
    Moreover, if $|z-w|\le \min\left\{\alpha^{-1/2}, \alpha^{-1}|z|^{-1} \right\}$, then
    $$
    \dist_\alpha(z,w) \le 4 |z-w| e^{\frac{\alpha|z|^2}2} \left(\alpha |z| + \sqrt\alpha \right).
    $$
\end{lemma}
\begin{proof}
    Recall that the reproducing kernel is $k_w(z)=e^{\alpha z\cdot \bar{w}}$.
    Making use of the defining property of the reproducing kernel, gives that
    \begin{equation}
        \begin{split}
            \dist_\alpha(z,w)^2 &= \langle k_z - k_w, k_z - k_w\rangle_{\ft_\alpha(\C^d)} \\
            & = k_z(z)-k_z(w)-k_w(z)+k_w(w) \\
            & = e^{\alpha|z|^2}- 2 \re (e^{\alpha z\cdot\bar{w}}) + e^{\alpha|w|^2}.
        \end{split}
    \end{equation}
    This proves the first part of the statement. For the second part, we denote $\delta \coloneqq w-z$ and get that
    \begin{equation}
        \begin{split}
            \dist_\alpha(z,w)^2 &= e^{\alpha|z|^2}- 2 \re (e^{\alpha z\cdot \bar{\delta}}) e^{\alpha|z|^2}+ e^{\alpha|z+\delta|^2}\\
          &= e^{\alpha|z|^2} \left(1 -2 \re(e^{\alpha z\cdot\bar{\delta}})+e^{\alpha(2\re(z\cdot\bar{\delta})+|\delta|^2)} \right).
        \end{split}
    \end{equation}
    The expression inside the brackets is equal to
    $$
    A=\left|e^{\alpha z\cdot \bar{\delta}} -1 \right|^2 + e^{2\alpha\re(z\cdot\bar{\delta})} (e^{\alpha|\delta|^2}-1).
    $$
    Note that $|e^u-1| \le (e-1)|u|$ for all $u \in \C$ with $|u|\le 1$.
    The assumption for the second statement implies that $\max\{\alpha |z\cdot\bar{\delta}|, \alpha|\delta|^2\}\le 1$. Hence, we can bound
    $$
    A \le (e-1)^2 \alpha^2 |z|^2 |\delta|^2 + e^{2\alpha\re(z\cdot \bar{\delta})} (e-1) \alpha|\delta|^2.
    $$
    Since $e^{2\alpha\re(z\cdot \bar{\delta})} \leq e^{2\alpha|z\cdot \bar{\delta}|} \leq e^2$,
    we obtain that 
    \begin{equation}
        \begin{split}
            \dist_\alpha(z,w)^2 = e^{\alpha|z|^2} A & \le e^{\alpha|z|^2}|\delta|^2
    \left((e-1)^2\alpha^2|z|^2+e^2(e-1)\alpha \right) \\
    & \le 13 e^{\alpha|z|^2}|\delta|^2 \left(\alpha^2|z|^2+\alpha \right),
        \end{split}
    \end{equation}
    which implies the second claim.
\end{proof}

An application of the Cauchy-Schwarz inequality yields Lipschitz estimates for functions in $\ft_\alpha(\Cd)$ with respect to the metric $\dist_\alpha$. We will need something a bit more specific concerning the real part of such functions.
\begin{lemma}\label{lem:realpartest}
 Let $G\in \ft_\alpha(\C^d)$ and suppose that $\zeta\in \C^d$ is such that $\re \, G(\zeta)=0$.
 Then it holds for all $\zeta'\in \C^d$ that 
 $$
 |\re \, G(\zeta')|\le \|G\|_{\ft_\alpha(\C^d)} \dist_\alpha(\zeta',\zeta).
 $$
\end{lemma}
\begin{proof}
 We rewrite
 $$
 \re \, G(\zeta') = \re \, G(\zeta')- \re\, G(\zeta) = \re\left(\langle G, k_{\zeta'} -k_{\zeta}\rangle_{\ft_\alpha(\C^d)} \right).
 $$
 The Cauchy-Schwarz inequality implies that 
 $$
 |\re \, G(\zeta')| \le \big| \langle G, k_{\zeta'} -k_{\zeta}\rangle_{\ft_\alpha(\C^d)}\big| \le \|G\|_{\ft_\alpha(\C^d)} \|k_{\zeta'}-k_{\zeta}\|_{\ft_\alpha(\C^d)}.
 $$
\end{proof}
The relevant case for our purposes is $d=2$. 
The following construction allows us to conceive  $F'\overline{F}$ as the restriction to $\R^2\subseteq \C^2$ of an entire function of two complex variables.
This will enable us to apply above estimates to $F'\overline{F}$, even though the function is not holomorphic (in the $1d$-sense) itself.

\begin{lemma}\label{lem:FprimeFbarest}

Suppose that $F\in\ft_\alpha(\C)$, and define 
$$
G(z_1,z_2):=F'(z_1+iz_2)F^*(z_1-iz_2).
$$
Then it holds for all $\beta>2\alpha$ that $G$ is an element of $\ft_\beta(\C^2)$, and satisfies the estimate
$$
\|G\|_{\ft_\beta(\C^2)} \le \frac{\beta^2}{(\beta-2\alpha)^{3/2}}\|F\|_{\ft_\alpha(\C)}^2.
$$
\end{lemma}

\begin{proof}
    First, recall that a function $F\in\ft_\alpha(\C)$ satisfies the pointwise growth estimate 
    $|F(w)| \le e^{\frac\alpha2 |w|^2} \|F\|_{\ft_\alpha(\C)}$. Moreover, a similar type of estimate is available for the derivative \cite[equation (2.5)]{haslinger}:
    \begin{equation}
            |F'(w)| \le \sqrt{\alpha (1+\alpha|w|^2)} e^{\frac\alpha2 |w|^2} \|F\|_{\ft_\alpha(\C)}.
    \end{equation}
    Let us denote $z=(z_1,z_2)^T\in\C^2$. By employing the two pointwise bounds on $|F|$ and $|F'|$, it follows that
    \begin{equation}
        \begin{split}
            \|G\|_{\ft_\beta(\C)}^2 &= \left(\frac\beta\pi\right)^2 \int_{\C^2} |F'(z_1+iz_2)|^2 |F(\overline{z_1}+i\overline{z_2})|^2 e^{-\beta|z|^2}\, \mbox{d}A(z)\\
            &\le \frac{\alpha\beta^2}{\pi^2}
            \|F\|_{\ft_\alpha(\C)}^4\int_{\C^2} (1+\alpha|z_1+iz_2|^2)
            e^{\alpha |z_1+iz_2|^2 + \alpha |z_1-iz_2|^2 -\beta|z|^2}
            \mbox{d}A(z)\\
            &= \frac{\alpha\beta^2}{\pi^2} \|F\|_{\ft_\alpha(\C)}^4\int_{\C^2} (1+\alpha|z_1+iz_2|^2)
            e^{-(\beta-2\alpha)|z|^2}
            \mbox{d}A(z) \\
            & \leq \frac{\alpha\beta^2}{\pi^2} \|F\|_{\ft_\alpha(\C)}^4 \int_{\C^2} (1+2\alpha|z|^2) e^{-(\beta-2\alpha)|z|^2}\,\mbox{d}A(z),
        \end{split}
    \end{equation}
   where we used that $|z_1+iz_2|^2+|z_1-iz_2|^2=2|z|^2$ and that $|z_1+iz_2| \leq \sqrt{2}|z|$.
   The integral on the right-hand side of the previous equation exists, provided that $\beta > 2 \alpha$. In this case, it evaluates to
$$
\int_{\C^2} (1+2\alpha|z|^2) e^{-(\beta-2\alpha)|z|^2}\,\mbox{d}A(z) = \frac{\pi^2(\beta+2\alpha)}{(\beta-2\alpha)^3}.
$$
Consequently, we arrive at
$$
\|G\|_{\ft_\beta(\C^2)}^2 \le \|F\|_{\ft_\alpha(\C)}^4 \frac{\alpha\beta^2 (\beta+2\alpha)}{(\beta-2\alpha)^3} \le \|F\|_{\ft_\alpha(\C)}^4 \frac{\beta^4}{(\beta-2\alpha)^3},
$$
where we used that $\beta > 2 \alpha$ in the last inequality.    
\end{proof}

We are prepared to prove the result stated at the beginning of this paragraph.
\begin{proof}[Proof of Proposition \ref{prop:zeroperturbation}]
    Let us denote $Q=|F|^2-|H|^2$ and $R_j=\delta_{\theta_j}[Q]$, $j\in\{1,2\}$. First observe, that with $\partial_z = \frac{1}{2}(\partial_x - i \partial_y)$ it holds that
\begin{equation}
    \partial_z Q = F'\overline{F}-H'\overline{H}.
\end{equation}
Now set $\tau=\theta_1-\theta_2$, and consider the linear system
$$
    \begin{pmatrix}
        \cos\theta_1 &\cos\theta_2\\
        \sin\theta_1 &\sin\theta_2
    \end{pmatrix}
    \begin{pmatrix}
        c_1\\ c_2
    \end{pmatrix}
    =
    \frac12 
    \begin{pmatrix}
        1\\ -i
    \end{pmatrix}
    $$
with its solution 
    $$
    c=
    \begin{pmatrix}
        c_1\\ c_2
    \end{pmatrix}
    =
    - \frac1{2\sin\tau} 
    \begin{pmatrix}
        \sin\theta_2 &-\cos\theta_2\\
        -\sin\theta_1 &\cos\theta_1
    \end{pmatrix}
    \begin{pmatrix}
        1\\-i
    \end{pmatrix},
    $$
    which has Euclidean length $|c|=| \sqrt{2} \sin\tau|^{-1}$.
    Using the definition of $R_j$ as a derivative into $e^{i\theta_j} \simeq \begin{pmatrix}
        \cos \theta_j \\ \sin \theta_j
    \end{pmatrix}$-direction yields
    $$
    R_j = \begin{pmatrix}
        \cos\theta_j\\ \sin\theta_j
    \end{pmatrix}
    \cdot \nabla Q = \begin{pmatrix}
        \re(e^{i\theta_j})\\ \im(e^{i\theta_j})
    \end{pmatrix}
    \cdot \begin{pmatrix}
        \partial_x Q\\ \partial_y Q
    \end{pmatrix}.
    $$
    In addition, one has
    $$
    2e^{i\theta_j} \partial_z Q = \re(e^{i\theta_j}) \partial_x Q + i \im(e^{i\theta_j}) \partial_x Q - i \re(e^{i\theta_j}) \partial_y Q + \im(e^{i\theta_j}) \partial_y Q.
    $$
    Therefore, we obtain the relation
    \begin{equation}\label{rjj}
    R_j = \re(2e^{i\theta_j} \partial_z Q).
    \end{equation}
    Hence,
    $$
    c_1 R_1 + c_2 R_2 = \frac12 \begin{pmatrix}
        1\\-i
    \end{pmatrix}
    \cdot \nabla Q  =\frac12 \left(\partial_x- i \partial_y \right) Q = \partial_z Q = F'\overline{F}-H'\overline{H}.
    $$
    An application of the Cauchy-Schwarz inequality gives
    \begin{equation}\label{eq:estitoR12}
    |(F'\overline{F}-H'\overline{H})(z_0)|\le |\sqrt{2} \sin\tau|^{-1} \sqrt{|R_1(z_0)|^2+ |R_2(z_0)|^2}.
    \end{equation}
    We continue by upper bounding the terms $|R_1(z_0)|$ and $|R_2(z_0)|$. To do so, we define for $j \in \{1,2\}$ functions $G_j : \C^2 \to \C$ via
    $$
    G_j(z_1,z_2) = 2e^{i\theta_j} (F'(z_1+iz_2)F^*(z_1-iz_2)-H'(z_1+iz_2)H^*(z_1-iz_2)).
    $$
    This definition implies that for all $z \in \C$ it holds that
    $$
    G_j(\re(z),\im(z)) = 2 e^{i\theta_j} (F'(z)\overline{F(z)} - H'(z)\overline{H(z)}).
    $$
    Using \eqref{rjj}, it follows that
    $$
    R_j(z) = \re (G_j(\re(z),\im(z))).
    $$
    Now let $\beta = 2\alpha + \varepsilon$. According to Lemma \ref{lem:FprimeFbarest} one has
    $$
    \| G_j\|_{\ft_\beta(\C^2)} \leq \frac{(2\alpha+\varepsilon)^2}{\varepsilon^{3/2}}
    \left( \|F\|_{\ft_\alpha(\C)}^2 +\|H\|_{\ft_\alpha(\C)}^2 \right).
    $$
    For $z_0 \in \C$ and $p_1,p_2 \in B_\eta(z_0)$ given as 
    in the assumption of the claim, set
    $$
    \zeta=(\re(p_j), \im(p_j))\in\C^2\quad\text{and} \quad \zeta'=(\re(z_0),\im(z_0))\in\C^2.
    $$
    The previous considerations in conjunction with Lemma \ref{lem:realpartest} yields
    \begin{equation}
        \begin{split}
            |R_j(z_0)| & \leq \| G_j\|_{\ft_\beta(\C^2)} \dist_\beta(\zeta',\zeta) \\
            & \leq \frac{(2\alpha+\varepsilon)^2}{\varepsilon^{3/2}}
    \left( \|F\|_{\ft_\alpha(\C)}^2 +\|H\|_{\ft_\alpha(\C)}^2 \right) \dist_\beta(\zeta',\zeta).
        \end{split}
    \end{equation}
We continue by upper bounding the term $\dist_\beta(\zeta',\zeta)$ via Lemma \ref{lem:estdist}. To do so, we observe that $|z_0|=|\zeta'|$ and that
$$
|\zeta' - \zeta| = |p_j-z_0| \leq \eta  \leq \min\left\{\beta^{-1/2}, \beta^{-1}|\zeta'|^{-1} \right\}.
$$
Therefore, the assumptions of Lemma \ref{lem:estdist} are fullfilled and we get
\begin{equation}
    \begin{split}
        \dist_\beta(\zeta',\zeta) & \leq 4 \eta e^{(\alpha - \frac{\varepsilon}{2})|z_0|^2} \left ( (2\alpha + \varepsilon)|z_0| + \sqrt{2\alpha + \varepsilon} \right ) \\
        & \leq 4 \eta e^{(\alpha - \frac{\varepsilon}{2})|z_0|^2} (2\alpha + \varepsilon + 1)(|z_0|+1).
    \end{split}
\end{equation}
Hence, by setting
$$
M \coloneqq 4 \frac{(2\alpha+\varepsilon)^2}{\varepsilon^{3/2}} (2\alpha + \varepsilon + 1) \left( \|F\|_{\ft_\alpha(\C)}^2 +\|H\|_{\ft_\alpha(\C)}^2 \right) ,
$$
we obtain
$$
|R_j(z_0)| \leq M \eta e^{(\alpha - \frac{\varepsilon}{2})|z_0|^2} (|z_0|+1).
$$
Plugging in the latter bound in \eqref{eq:estitoR12} yields
$$
|(F'\overline{F}-H'\overline{H})(z_0)|\le M  \frac{(|z_0|+1) e^{(\alpha+\frac\varepsilon2)|z_0|^2}}{|\sin(\theta_1-\theta_2)|}  \eta,
$$
as announced.
\end{proof}

\subsection{Perturbation of Liouville sets}
Next, we require a lemma which indicates, that the property of being a Liouville set is invariant under a certain type of perturbation.

\begin{lemma}\label{lma:liouville_invariance}
    Let $\Lambda \subseteq \C$ be a Liouville set for $\ft_\alpha(\C)$, and let
    $$
    f(z)=e^{-\gamma |z|^2}, \quad \gamma > \alpha/2.
    $$
    If $\Gamma = (\gamma_\lambda)_{\lambda \in \Lambda} \subseteq \C$ is $f$-close to $\Lambda$, then $\Gamma$ is a Liouville set for $\ft_\alpha(\C)$.
\end{lemma}
\begin{proof}
    Let $C>0$ and $F \in \ft_\alpha(\C)$ such that $|F(\gamma_\lambda)| \leq C$ for all $\lambda \in \Lambda$. We have to show that $F$ is a constant function. To this end, we observe that for all $\lambda\in \Lambda$
    $$
    |F(\lambda)| \le |F(\gamma_\lambda)-F(\lambda)| + |F(\gamma_\lambda)| \le \dist_\alpha(\lambda,\gamma_\lambda) \|F\|_{\ft_\alpha(\C)} + C.
    $$
    For $R>0$ sufficiently large, we have according to Lemma \ref{lem:estdist}  that 
    \begin{equation}\label{eq:lambda_gamma}
        \mathrm{dist}_\alpha(\lambda,\gamma_\lambda) \leq 4|\lambda - \gamma_\lambda|e^{\frac{\alpha}{2}|\lambda|^2}(\alpha |\lambda| + \sqrt{\alpha})
    \end{equation}
    for all $\lambda \in \Lambda$ which satisfy $|\lambda| > R$. The closeness assumption on $\Gamma$ implies that the right-side in equation \eqref{eq:lambda_gamma} is bounded. Hence $F$ is bounded on $\Lambda$, which is a Liouville set by assumption. Consequently, $F$ is a constant function.
\end{proof}

\subsection{Integrability lemma}

Finally, we require a technical lemma concerning the difference $FH'-F'H$ with $F,H$ elements in the Fock space.

\begin{lemma}\label{lem:Fprimegrowth}
    Let $\alpha > 0$, and let $F,H \in \ft_\alpha(\C)$. Then $FH'-F'H \in \ft_{2\alpha}(\C)$.
\end{lemma}

\begin{proof}
We only have to show that $FH'-F'H\in L^2(\C, \mbox{d}\mu_{2\alpha})$ as this function is obviously entire.
It is well known, that if $H \in \ft_\alpha(\C)$, then the function $\tilde H (z) \coloneqq H'(z) - \alpha \overline{z} H(z)$ belongs to the polyanalytic Fock space of order two, i.e., $\tilde H \in L^2(\C,\mbox{d}\mu_\alpha)$, and $\partial_{\bar{z}}^2 \tilde H = 0$, where $\partial_{\bar{z}} = \frac{1}{2}(\partial_x + i \partial_y)$ \cite{abreupolyanalytic}. Now observe, that
\begin{equation}\label{eq:finid_intlemma}
    \begin{split}
         (FH'-F'H)(z) &= F(z) \left(\alpha \overline{z} H(z) + \tilde H(z) \right)
        - \left(\alpha\overline{z}F(z) + \tilde F(z) \right) H(z) \\
        & = F(z)  \tilde H(z) 
        -\tilde F(z)  H(z). 
    \end{split}
\end{equation}
By invoking \eqref{eq:growth_fock_space}, we get that 
\begin{align*}
    \int_\C |F(z)  \tilde H(z)|^2 \,\mbox{d}\mu_{2\alpha}(z) 
    &= \int_\C |F(z)|^2 e^{-\alpha|z|^2} |\tilde H(z)|^2 e^{-\alpha|z|^2}\, \mbox{d}A(z) \\
    &\le 
    \|F\|_{\ft_\alpha(\C)}^2 \| \tilde H\|_{L^2(\C, \mbox{d}\mu_\alpha)}^2 < \infty.
\end{align*}
With this (and by interchanging roles of $F$ and $H$) we find that 
both terms on the right hand side of \eqref{eq:finid_intlemma} are members of $L^2(\C, \mbox{d}\mu_{2\alpha})$, and as a consequence $FH'-F'H$ has the same property. This concludes the proof.
\end{proof}

\section{Proofs of the main results}

\subsection{Deterministic perturbation of Liouville sets}\label{subsec:deterministicfock}

We are prepared to prove the first main result of the present exposition.

\begin{proof}[Proof of Theorem \ref{thm:Main1}]
First we observe that we can assume without loss of generality that 
$
    \Lambda \cap B_1(0) = \emptyset.
$
Indeed, since the property of being a Liouville set is invariant under removing a bounded set of points we may just consider $\Lambda\setminus B_1(0)$ instead of $\Lambda$.
Recall that by assumption, condition \eqref{cond:anglesgeneral} is satisfied. That is, 
$$
\exists \beta\in (0, \gamma-2\alpha), \, \exists L>0: \quad \frac{|\lambda|e^{-\beta|\lambda|^2}}{\varphi(a_\lambda,b_\lambda,c_\lambda)} \leq L, \quad \lambda\in\Lambda.
$$
In particular, all of the triangles $\Delta(a_\lambda,b_\lambda,c_\lambda)$ are non-degenerate. Without loss of generality, we can assume that for each $\lambda \in \Lambda$, the angle $\varphi(a_\lambda,b_\lambda,c_\lambda)$ is the acute angle enclosed between the lines 
$c_\lambda+\R(a_\lambda-c_\lambda)$ and $c_\lambda+\R(b_\lambda-c_\lambda)$. This can always be achieved by means of a pairwise interchange of the elements in $A,B,C$. It follows from Lemma \ref{lma:liouville_invariance} and the assumption on $\Lambda$, that $C$ is a Liouville set for $\ft_{4\alpha}(\C)$. This implies that $C$ is a uniqueness set for $\ft_{4\alpha}(\C)$ and for $\ft_{\alpha}(\C)$.
Suppose $F,H\in\ft_\alpha(\C)$ are such that 
    $$
    |F(\lambda)|= |H(\lambda)|, \quad \lambda \in A\cup B \cup C.
    $$
    We want to show that $F\sim H$. As per Proposition \ref{lma:uniqueness_lemma} it suffices to prove that 
    $$
    G= FH(FH'-F'H)
    $$
     vanishes identically. Utilizing Lemma \ref{lem:Fprimegrowth} implies that $G\in \ft_{4\alpha}(\C)$. Consequently, it suffices to show that 
$G$ is bounded on $C$ and that 
    $$\inf_{\lambda\in\Lambda} |G(c_\lambda)|=0.$$

To show this, we proceed as follows. We use $K_1,K_2,\ldots$ for positive constants which crucially do not depend on $\lambda\in\Lambda$.
As $C$ is $f$-close to $\Lambda$, there exists $K_1>0$ such that 
\begin{align*}
\big| |c_\lambda|^2-|\lambda|^2\big| &= (|c_\lambda|+|\lambda|)  \big||c_\lambda|-|\lambda|\big| \\
&\le 
(2|\lambda|+|c_\lambda-\lambda|) |c_\lambda-\lambda| \\ &\le (2|\lambda|+ K_1 e^{-\gamma|\lambda|^2}) K_1 e^{-\gamma|\lambda|^2}.
\end{align*}
In particular, there exists $K_2$ (independent from $\lambda\in\Lambda$) such that 
\begin{equation}\label{K22}
    \big| |c_\lambda|^2-|\lambda|^2\big|\le K_2.
\end{equation}
Consequently, 
$$
f(\lambda) = e^{-\gamma|\lambda|^2} = e^{-\gamma|c_\lambda|^2} e^{\gamma(|c_\lambda|^2-|\lambda|^2)}
\le e^{-\gamma|c_\lambda|^2} e^{\gamma K_2}.
$$
Since $A,C$ are $f$-close to $\Lambda$, it follows that there exists $K_3>0$ such that 
$$
|a_\lambda-c_\lambda|\le |a_\lambda-\lambda| + |c_\lambda-\lambda| \le K_3 f(\lambda) \le
K_3 e^{\gamma K_2} e^{-\gamma|c_\lambda|^2}
,\quad \lambda\in \Lambda.
$$
Thus, $A$ is $f$-close to $C$. We can argue in the same way for $B$ instead of $A$ to get that $B$ is $f$-close to $C$. Therefore, both $A$ and $B$ are both $f$-close to $C$, i.e., there exists $\kappa>0$ such that 
$$
\max\left\{ |a_\lambda-c_\lambda| , |b_\lambda-c_\lambda| \right\} \le \eta(\lambda), \quad \lambda\in\Lambda, 
$$
where $\eta(\lambda)\coloneqq \kappa e^{-\gamma |c_\lambda|^2}$.

We seek for an application of Proposition \ref{prop:zeroperturbation}. To this end, we introduce for every $\lambda \in\Lambda$ angles $\theta_{1,\lambda}$ and $\theta_{2,\lambda}$ by
    \begin{equation*}
    \theta_{1,\lambda}=  \arg(a_\lambda-c_\lambda),\quad 
    \theta_{2,\lambda} = \arg(b_\lambda-c_\lambda).
    \end{equation*}
    Note that 
    $$\varphi(a_\lambda,b_\lambda,c_\lambda) = \pm (\theta_{1,\lambda}-\theta_{2,\lambda}) \mod 2\pi,\quad \lambda\in\Lambda.$$
    By assumption, we have that $Q=|F|^2-|H|^2$ vanishes at the endpoints of the segment connecting $c_\lambda$ and $a_\lambda$. As per Rolle's theorem there exists a point $p_{1,\lambda}$ on the segment where the derivative satisfies $\delta_{\theta_{1,\lambda}} [Q](p_{1,\lambda})=0$.
    One can argue analogously for $b_\lambda$ instead of $a_\lambda$.
    Hence, we get that
    $$\exists p_{1,\lambda},p_{2,\lambda}\in B_{\eta(\lambda)}(c_\lambda):
    \quad 
    \delta_{\theta_{j,\lambda}}[Q](p_{j,\lambda}) = 0 , \, j\in\{1,2\}.
    $$
    Let $\varepsilon :=\gamma-2\alpha-\beta>0$, and let
     $R_0\ge1$ be chosen in such a way that for all $c_\lambda\in C$ with $|c_\lambda|>R_0$ one has
    $$
    \eta(\lambda)\le \min\left\{ \left( 2\alpha+\varepsilon \right )^{-1/2}, \left (2\alpha+\varepsilon \right)^{-1} \cdot |c_\lambda|^{-1}\right\}.
    $$
     Clearly, this condition is satisfied if $R_0$ is sufficiently large.
    In the following, we assume that $c_\lambda\in C$ satisfies $|c_\lambda|>R_0$.
    Proposition \ref{prop:zeroperturbation} implies that there exists a constant $M=M(\alpha, \varepsilon, \|F\|_{\ft_\alpha(\C)},\|H\|_{\ft_\alpha(\C)})$ such that
\begin{equation}
    \begin{split}
        |(F'\overline{F}-H'\overline{H})(c_\lambda)| & \leq M \frac{(|c_\lambda|+1)e^{(\alpha +\frac\varepsilon2)|c_\lambda|^2}}{|\sin(\theta_{1,\lambda} - \theta_{2,\lambda})|} \eta(\lambda) \\
        & \leq 2 \kappa M \frac{|c_\lambda|e^{(\alpha +\frac\varepsilon2)|c_\lambda|^2}}{|\sin(\varphi(a_\lambda,b_\lambda,c_\lambda))|} e^{-\gamma|c_\lambda|^2} \\
        & \leq \pi \kappa M \frac{|c_\lambda|e^{(\alpha +\frac\varepsilon2)|c_\lambda|^2}}{\varphi(a_\lambda,b_\lambda,c_\lambda)} e^{-\gamma|c_\lambda|^2} \\
        & \leq \pi \kappa M L \frac{|c_\lambda|}{|\lambda|} e^{(\alpha +\frac\varepsilon2)|c_\lambda|^2} e^{-\gamma|c_\lambda|^2} e^{\beta|\lambda|^2}.
    \end{split}
\end{equation}
In the third inequality we used that
$
\sin\vartheta \ge \frac2\pi \vartheta
$
for all $\vartheta \in [0,\frac{\pi}{2}]$, and in the fourth inequality we used the assumption on $\varphi(a_\lambda,b_\lambda,c_\lambda)$.
Since
$$
    |c_\lambda| \le |\lambda| + |\lambda-c_\lambda| \le |\lambda|+K_1 \le (K_1+1)|\lambda|,
$$
it follows that there exists a constant $K_4 > 0$ such that
$$
|(F'\overline{F}-H'\overline{H})(c_\lambda)| \leq K_4 e^{(\alpha +\frac\varepsilon2)|c_\lambda|^2} e^{-\gamma|c_\lambda|^2} e^{\beta|\lambda|^2}.
$$
By noticing that $|F(z)|^2e^{-\alpha|z|^2}$ is bounded in $\C$, and that $|F(c_\lambda)| = |H(c_\lambda)|$, it follows that there exists a constant $K_5>0$ such that
\begin{equation}
    \begin{split}
        |G(c_\lambda)| &= |F(c_\lambda)| |H(c_\lambda)| |(FH'-F'H)(c_\lambda)| \\
        & = |(|H|^2 FH' - |F|^2 F' H)(c_\lambda)| \\
        &= |F(c_\lambda)|^2 |(H'\overline{H}-F'\overline{F})(c_\lambda)| \\
        & \leq K_5 e^{\alpha|c_\lambda|^2} e^{(\alpha +\frac\varepsilon2)|c_\lambda|^2} e^{-\gamma|c_\lambda|^2} e^{\beta|\lambda|^2}.
    \end{split}
\end{equation}
By the choice of $\varepsilon$ we have that
$$
2\alpha + \frac{\varepsilon}{2} - \gamma = - \beta - \frac{\varepsilon}{2}.
$$
This, in combination with equation \eqref{K22} yields
$$
e^{\alpha|c_\lambda|^2} e^{(\alpha +\frac\varepsilon2)|c_\lambda|^2} e^{-\gamma|c_\lambda|^2} e^{\beta|\lambda|^2} = e^{-\frac\varepsilon2 |c_\lambda|^2} e^{\beta(|\lambda|^2-|c_\lambda|^2)} \leq e^{-\frac\varepsilon2 |c_\lambda|^2} e^{ \beta K_2}.
$$
As a result, there exists a constant $K_6>0$ such that
$$
|G(c_\lambda)| \leq K_6 e^{-\frac\varepsilon2 |c_\lambda|^2}.
$$
Since $C$ is necessarily unbounded we obtain that $\inf_{\lambda\in\Lambda} |G(c_\lambda)| =0.$
On the other hand, we have that 
    $$
    \sup_{\lambda\in\Lambda} |G(c_\lambda)| \le \max\left\{\sup_{|z|\le R_0} |G(z)|, K_6  \right\} < \infty.
    $$
    This finishes the proof.
\end{proof}

\begin{remark}[On the sharpness of Theorem \ref{thm:Main1}]\label{rem:sharpness}
Theorem \ref{thm:Main1} relies on an interplay between two conditions: $f$-closeness of $A,B,C$ to a Liouville set $\Lambda$ with respect to $f(z)=e^{-\gamma |z|^2}$, and the angle condition \eqref{cond:anglesgeneral}. We emphasize that dropping one of the two conditions implies that $\mathcal{U} = A \cup B \cup C$ is in general not a uniqueness set for phase retrieval in Fock space anymore. To see this, we assume for simplicity that $\Lambda$ is a lattice satsfying $s(\Lambda) < \frac{\pi}{4\alpha}$. According to Theorem \ref{thm:Main3}, $\Lambda$ is a Liouville set for $\ft_{4\alpha}(\C)$.

First, consider the case where $A,B,C$ are merely $f$-close to $\Lambda$ but \eqref{cond:anglesgeneral} is not fulfilled. In this case, one can choose $A,B,C$ in such a way that
\begin{equation}\label{eq:incl}
    \mathcal{U} \subseteq e^{i\theta} (\R \times \nu \Z)
\end{equation}
for suitable $\nu >0$ and $\theta \in \R$. That is, $\mathcal{U}$ is contained in a set of equidistant parallel lines. According to \cite[Theorem 1]{alaifari2020phase}, $\mathcal{U}$ is not a uniqueness set for the phase retrieval problem in $\ft_\alpha(\C)$.

On the other hand, suppose that $A,B,C$ have the property that $\mathcal{U} \subseteq \Gamma$, where $\Gamma \subseteq \C$ is a lattice, and that condition \eqref{cond:anglesgeneral} is fullfilled. Such a choice for $A,B,C$ can be done without $A,B,C$ being $f$-close to $\Lambda$. For instance, we can assume that $A,B,C$ are merely uniformly close to $\Lambda$. Since lattices are never uniqueness sets for the phase retrieval problem in $\ft_\alpha(\C)$, it follows that $\mathcal{U}$ is not a uniqueness set for phase retrieval in Fock space, although condition  \eqref{cond:anglesgeneral} is satisfied.
\end{remark}

\subsection{Random perturbation of Liouville sets}\label{subsec:randomsampfock}

This section is devoted to the proof of Theorem \ref{thm:Main4}, which states that the union of three random perturbations of a Liouville set forms a uniqueness set for the phase retrieval problem almost surely.

We start with the observation that three randomly picked points $A,B,C$ in the unit disk (uniformly distributed) are almost surely noncollinear. In particular, the resulting triangle $\Delta(A,B,C)$ is almost surely non-degenerate. 
We require a quantitative version of this observation.

\begin{lemma}\label{lem:3randompts}
    Let $\Omega\subseteq \C$ be a disk in the complex plane.
    Let $A,B,C$ be independent and identically distributed complex random variables, uniformly distributed on $\Omega$.
    Then it holds for all $\varepsilon > 0$ that 
    $$
    \mathbb{P}[\varphi(A,B,C)< \varepsilon] \le 4 \varepsilon.
    $$
\end{lemma}
\begin{proof}
Without loss of generality, we assume that $\Omega=\mathbb{D}$.
It is obvious that $A,B,C$ are pairwise distinct with probability $1$.
Let $a,b\in\mathbb{D}$ be arbitrary, and let $L(a,b) = a+\R(b-a)$ denote the line through $a$ and $b$.
Moreover, let $\delta > 0$ and  let $$S_\delta(a,b) =\{z\in \mathbb{D}: \dist(L(a,b),z)<\delta\}.$$
We note that if $c\in \mathbb{D} \setminus S_\delta(a,b)$, then 
the triangle $\Delta(a,b,c)$ has an acute angle $\psi$ at either $a$ or $b$ (or both) which satisfies $\tan\psi \ge \frac\delta2$.
As $\arctan(x)\ge 2 \arctan(1/2) x$ for $x\in[0, \frac{1}{2} ]$, we get that 
\begin{equation}
\varphi(a,b,c) \ge \psi \ge \arctan(\delta/2) \ge \arctan(1/2) \delta \ge 0.4 \delta.
\end{equation}
\begin{center}
\begin{figure}[ht]
\begin{tikzpicture}[scale=1.5]
\draw [fill=lightgray,dotted] (0,0) circle (1);
\draw [thick,dash dot] (-1.1,0.1) -- (1.1,0.1);
\draw [thick] (-1.3,0.4) -- (1.3,0.4);
\draw [thick] (-1.3,-0.2) -- (1.3,-0.2);
\draw[blue, fill=blue] (0.7,0.5) circle [radius=1pt];
\draw[red, fill=red] (-0.8,0.1) circle [radius=1pt];
\draw[red, fill=red] (-0.5,0.1) circle [radius=1pt];
\node[blue] at (0.7,0.65) {c};
\node[red] at (-0.8,-0.04) {a};
\node[red] at (-0.5,-0.04) {b};
\draw [<->] (1.2,0.4) -- (1.2,-0.2);
\node at (1.33,0.1) {$2\delta$};
\end{tikzpicture}
\end{figure}
\end{center}
Since $
    \mathbb{P}[C \in  S_\delta(a,b)] \le  \frac4\pi \delta,
$
it follows that for $\delta = \frac{\varepsilon}{0.4}$, we have
$$
\mathbb{P}[ \varphi(a,b,C) \geq \varepsilon ] = \mathbb{P}[ \varphi(a,b,C) \geq 0.4 \delta ] \geq \mathbb P [C \not \in S_\delta(a,b) ] \geq 1 - \frac{4}{\pi}\delta.
$$
Consequently,
$$
\mathbb P [\varphi(a,b,C) < \varepsilon] \leq \frac{4 \varepsilon}{0.4 \pi} \leq 4 \varepsilon,
$$
and therefore $\mathbb P [\varphi(A,B,C) < \varepsilon] \leq 4 \varepsilon$.
\end{proof}

Notice, that Lemma \ref{lem:3randompts} makes a probabilistic statement on the acute angle $\varphi(a,b,c)$ of the triangle $\Delta(a,b,c)$ in a disk. Related probabilistic problems on acute angles in triangles were studied, for instance, in \cite{trii1,trii2}.
We are equipped to prove the uniqueness result regarding three random perturbations of a Liouville set for $\ft_{4\alpha}(\C)$.

\begin{proof}[Proof of Theorem \ref{thm:Main4}]
We show that the statement holds true when the assumption that $\Lambda$ has finite density is replaced by the weaker condition that $\Lambda$ is countable and that
\begin{equation}
    \exists \beta\in (0,\gamma-2\alpha):\quad \sum_{\lambda\in\Lambda} |\lambda| e^{-\beta|\lambda|^2} <\infty.
\end{equation}

Clearly, we have that $A=\{Z_{\lambda,1} : \lambda\in\Lambda\}$, $B=\{Z_{\lambda,2} : \lambda\in\Lambda\}$ and $C=\{Z_{\lambda,3} : \lambda\in\Lambda\}$ are noncollinear almost surely.
In order to deduce the desired assertion from Theorem \ref{thm:Main1}, it remains to verify that condition \eqref{cond:anglesgeneral} holds almost surely, i.e., to show that the sequence
$$  \frac{|\lambda|e^{-\beta|\lambda|^2}}{\varphi(A_\lambda,B_\lambda,C_\lambda)}, \quad \lambda  \in\Lambda
$$
is bounded with probability $1$.
To this end, we define the events
$$
E\coloneqq \left\{ \sup_{\lambda\in\Lambda} \frac{|\lambda| e^{-\beta|\lambda|^2}}{\varphi(A_\lambda, B_\lambda, C_\lambda)} < \infty\right\},
$$
as well as 
$$
E_L \coloneqq \left\{ \sup_{\lambda\in\Lambda}  \frac{|\lambda| e^{-\beta|\lambda|^2}}{\varphi(A_\lambda, B_\lambda, C_\lambda)} \le L\right\}, \quad L>0.
$$
By continuity of measure, we have that $\mathbb{P}[E]=\lim_L \mathbb{P}[E_L]$.
With Lemma \ref{lem:3randompts} we get
\begin{align*}
    \mathbb{P}[E_L] &= 1 - \mathbb{P} \left [\exists \lambda\in\Lambda: \, \frac{|\lambda| e^{-\beta|\lambda|^2}}{\varphi(A_\lambda, B_\lambda, C_\lambda)} > L \right]\\
    &= 1 - \mathbb{P}\left [\exists \lambda\in\Lambda: \varphi(A_\lambda,B_\lambda,C_\lambda) < \frac{|\lambda|e^{-\beta|\lambda|^2}}{L} \right]\\
    &\ge 1- \sum_{\lambda\in\Lambda} \mathbb{P}\left [\varphi(A_\lambda,B_\lambda,C_\lambda) < \frac{|\lambda|e^{-\beta|\lambda|^2}}{L} \right ]\\
    &\ge 1 - \frac4{L} \sum_{\lambda\in\Lambda} |\lambda|e^{-\beta|\lambda|^2}.
\end{align*}
Recall that $\sum_{\lambda\in\Lambda} |\lambda|e^{-\beta|\lambda|^2} < \infty$ by assumption. Therefore, the right hand side converges to $1$ as $L\to\infty$. Thus, 
$\mathbb{P}[E] = \lim_L \mathbb{P}[E_L] = 1$,
meaning that \eqref{cond:anglesgeneral} holds almost surely, and we are done.
\end{proof}

\begin{remark}
    Theorem \ref{thm:Main4} could have been stated slightly more generally:
    the proof only uses that $(Z_{\lambda,1}, Z_{\lambda,2}, Z_{\lambda_3})$ are independent for each $\lambda\in\Lambda$, 
    which is weaker than the condition demanded in Theorem \ref{thm:Main4} (i.e., that the sampling points are picked independently across the full index set $\Lambda\times\{1,2,3\}$).
    Corresponding statements hold true for Theorems \ref{thm:Main5}, \ref{thm:Main6} and \ref{thm:Main7}.
\end{remark}

\subsection{Sub-classes with symmetry properties}

The contents presented in this section provide the foundation for proving the Gabor phase retrieval results in the spaces of real-valued and even real-valued functions. We will primarily investigate two classes of entire functions which obey certain symmetry conditions, and are defined by 
$$
\mathcal{S}_- := \{ F\in \mathcal{O}(\C): F(-z) = F(z)\}, \quad \mathcal{S}^* := \{F\in\mathcal{O}(\C): F^* = F\}.
$$
Further, we denote their intersection by $\mathcal{S}_-^* \coloneqq \mathcal{S}_- \cap \mathcal{S}^*$.
The symmetry classes $\mathcal{S}_-$ and $\mathcal{S}^*$ arise naturally as the images of suitable sub-classes of $L^2(\R)$ under the Bargmann transform $\mathcal{B} : \lt \to \ft_\pi(\C)$.
Recall that $L^2(\R,\R)\subseteq L^2(\R)$ denotes the subspace of real-valued functions. By $L^2_e(\R)\subseteq L^2(\R)$ we denote the subspace of even functions. With this, we have the relations  
    $$
    \mathcal{B}(L^2(\R,\R)) = \ft_\pi(\C) \cap \mathcal{S}^*, \quad 
    \text{and} \quad 
    \mathcal{B}(L^2_e(\R)) = \ft_\pi(\C) \cap \mathcal{S}_-.
    $$
    Hence, for the intersection consisting precisely of all even and real-valued functions, denoted by $L^2_e(\R,\R)$, we have that 
    $$
    \mathcal{B}(L^2_e(\R,\R)) = \ft_\pi (\C)\cap \mathcal{S}_-^*.
    $$
We proceed by introducing two variants of Liouville sets which are geometrically linked to the subclasses $\mathcal{S}^*$ and $\mathcal{S}^*_-$ in Fock space.

\begin{definition}
    Let $V\subseteq \mathcal{O}(\C)$.
    We say that $\Lambda\subseteq\C$ is a  
    $\nicefrac12$-Liouville set  ($\, \nicefrac14$-Liouville set, resp.) for $V$ if 
    $$\Lambda\cup \overline{\Lambda}\, \quad \big[\Lambda \cup \overline{\Lambda} \cup (-\Lambda) \cup (-\overline{\Lambda}), \,\text{resp.}\big]
    $$
    is a Liouville set for $V$.
\end{definition}
Simple examples of $\nicefrac12$- and $\nicefrac14$-Liouville sets for Fock spaces are obtained by restricting suitable lattices to the upper (or lower) half plane or to one of the quadrants, respectively.
For instance, if $\Gamma=a\Z + i b\Z$ with $s(\Gamma)=ab<\frac\pi\alpha$, then it is easy to see that $\Lambda = \Gamma\cap \overline{\mathbb{H}_+}$ is a $\nicefrac12$-Liouville set for $\ft_\alpha(\C)$. Indeed, as per Theorem \ref{thm:Main3} we have that 
$$
\Gamma = \Lambda \cup \overline{\Lambda} 
$$
is a Liouville set for $\ft_\alpha(\C)$. Similarly, if one only takes the points of $\Gamma$ which are located in the closed first quadrant one ends up with a $\nicefrac14$-Liouville set for $\ft_\alpha(\C)$.
At this junction, we can formulate two uniqueness theorems for the phase retrieval problem in $\ft_\pi(\C) \cap\mathcal{S}^*$ and $\ft_\pi(\C) \cap\mathcal{S}^*_-$ which employ the notions of $\nicefrac12$-Liouville set and $\nicefrac14$-Liouville set. The first one reads as follows.

\begin{proposition}\label{prop:1214}
    Let $\alpha  > 0$, let $f:\C\to [0,\infty)$ be given by $f(z) = e^{-\gamma|z|^2}$, $\gamma>2\alpha$, and let $\Lambda \subseteq \C$. Suppose that $A,B,C \subseteq \C$ are $f$-close to $\Lambda$, and that 
    \begin{equation}\label{eq:bd_condition}
        \exists \beta\in (0,\gamma-2\alpha),\, \exists L>0: \quad \frac{|\lambda|e^{-\beta|\lambda|^2}}{\varphi(a_\lambda,b_\lambda,c_\lambda)} \le L, \quad \lambda\in\Lambda.
    \end{equation}
    Then the following holds.
    \begin{enumerate}
        \item If $\Lambda$ is a $\nicefrac12$-Liouville set for $\ft_{4\alpha}(\C)$, then every set $\mathcal{U} \subseteq \C$ with 
\begin{equation}
    \mathcal{U} \cup \overline{\mathcal{U}} \supseteq   A \cup B \cup C
\end{equation}
is a uniqueness set for the phase retrieval problem in $\mathcal{F}_\alpha(\C)\cap \mathcal{S}^*$.
\item If $\Lambda$ is a $\nicefrac14$-Liouville set for $\ft_{4\alpha}(\C)$, every set $\mathcal{U}\subseteq \C$ with 
\begin{equation}
    \mathcal{U} \cup \overline{\mathcal{U}} \cup (-\mathcal{U}) \cup (-\overline{\mathcal{U}}) \supseteq A \cup B \cup C
\end{equation}
is a uniqueness set for the phase retrieval problem in $\mathcal{F}_\alpha(\C)\cap \mathcal{S}_-^*$.
    \end{enumerate}
\end{proposition}
\begin{proof}
    To prove the first assertion, we let $F,H \in \mathcal{F}_\alpha(\C)\cap \mathcal{S}^*$ such that
    \begin{equation}\label{eq:FHLC}
        |F(\lambda)| = |H(\lambda)|, \quad \lambda \in \mathcal{U}.
    \end{equation}
    We have to show that $F \sim H$. To this end, we observe that the assumption $F,H \in \mathcal{S}^*$ gives $|F(z)|=|F(\overline{z})|$ and $|H(z)| = |H(\overline{z})|$ for every $z \in \C$. It follows from \eqref{eq:FHLC}, that
    $$
    |F(\lambda)| = |H(\lambda)|, \quad \lambda \in \mathcal{U}\cup \overline{\mathcal{U}}.
    $$
    Therefore, it suffices to prove that $\mathcal{U}\cup \overline{\mathcal{U}}$ is a uniqueness set for the phase retrieval problem in $\ft_\alpha(\C)$. To do so, we observe that the assumption on $\mathcal{U}$ implies that
    $$
    \mathcal{U} \cup \overline{\mathcal{U}} \supseteq  (A \cup \overline{A}) \cup (B \cup \overline{B}) \cup(C\cup \overline{C}).
    $$
    As $\Lambda$ is assumed to be a $\nicefrac12$-Liouville set for $\ft_{4\alpha}(\C)$, we have that $\Lambda'=\Lambda\cup \overline{\Lambda}$ is a Liouville set for $\ft_{4\alpha}(\C)$. 
    Moreover, we can extract subsequences 
    $A'=(a_{\lambda'})_{\lambda'\in\Lambda'}, B'=(b_{\lambda'})_{\lambda'\in\Lambda'}, C'=(c_{\lambda'})_{\lambda'\in\Lambda'} \subseteq  \mathcal{U}\cup \overline{\mathcal{U}}$, which are $f$-close to $\Lambda'$ and such that condition \eqref{eq:bd_condition} holds with $\Lambda$ replaced by $\Lambda'$. The statement now follows by applying Theorem \ref{thm:Main1}.

The second assertion follows analogously to the first assertion.
\end{proof}

\subsection{Gabor phase retrieval via random perturbations}\label{subsec:gaborrandomrestriction}

This section is dedicated to the proof of Theorem \ref{thm:Main6} and Theorem \ref{thm:Main7}, which state that two resp. one random perturbations of a lattice forms a uniqueness set for the Gabor phase in $L^2(\R,\R)$ resp. $L^2_e(\R,\R)$ almost surely. To begin with, we require a result which asserts that $\Delta(A,B,\overline{A})$ is generically non-degenerate when $A,B$ are picked at random in a disk centered on the real axis.

\begin{lemma}\label{lem:2randompts}
    Let $\Omega\subseteq\C$ be a disk centered at the real line. 
    Moreover, let $A,B$ be a pair of independent complex random variables, both uniformly distributed on $\Omega$.
    Then it holds for all $\varepsilon>0$ that 
    $$
    \mathbb{P}[\varphi(A,B,\overline{A})< \varepsilon] \le 4  \varepsilon.
    $$
\end{lemma}
\begin{proof}
    Without loss of generality, we may assume that $\Omega=\mathbb{D}$.
    For $a\in\mathbb{D}$, and for $\delta>0$, we denote
    $$
    S_\delta(a) \coloneqq \{z\in\mathbb{D}:\, |\re{(z)} - \re{(a)}|< \delta\}.
    $$
    Notice that with probability one it holds that $a \in \mathbb{D}\setminus \R$. In this case, an analogous argument as in the proof of Lemma \ref{lem:3randompts} shows that for all $b\in \mathbb D \setminus S_\delta(a)$, we have
    $$
    \varphi(a,b,\overline{a}) \ge 0.4 \delta.
    $$
Since 
    $
    \mathbb{P}[B\in S_\delta(a)] \le \frac4\pi \delta,
    $
it follows that
$$
\mathbb{P}[\varphi(a,B,\overline{a}) < 0.4 \delta] \leq 1-\mathbb{P}[B \not \in S_\delta(a)] \leq \frac{4}{\pi} \delta.
$$
Choosing $\delta = \frac{\varepsilon}{0.4}$ shows that
$$
\mathbb{P}[\varphi(a,B,\overline{a}) < \varepsilon] \leq 4 \varepsilon,
$$
and therefore $\mathbb{P}[\varphi(A,B,\overline{A}) < \varepsilon] \leq 4 \varepsilon$
\end{proof}

With this, we are ready to prove that two random perturbations of a sufficiently dense lattice $\Lambda\in\mathcal{L}$ yield a uniqueness set for Gabor phase retrieval in $L^2(\R,\R)$ almost surely.

\begin{proof}[Proof of Theorem \ref{thm:Main6}]
We identify $\R^2$ with the complex plane by virtue of $(x,y)^T \simeq x+iy$.
Recall that $\mathcal{B} (L^2(\R,\R)) = \ft_{\pi}(\C)\cap \mathcal{S}^*$.
We will show that $\mathcal{U}$ forms a uniqueness set for the phase retrieval problem in $\ft_\pi(\C)\cap \mathcal{S}^*$ almost surely, which implies the statement.
To do so, denote by $\mathbb{H}_+$ the open upper halfplane and define $\Lambda_+\coloneqq \Lambda \cap \overline{\mathbb{H}_+}$.
As $\Lambda\in\mathcal{L}$, we have that $\Lambda_+\cup\overline{\Lambda_+}=\Lambda$.
Therefore, as $\Lambda$ is a Liouville set for $\ft_{4\pi}(\C)$, we have that $\Lambda_+$ is a $\nicefrac12$-Liouville set for $\ft_{4\pi}(\C)$.
For $\lambda \in \Lambda_+$ we define 
$$
A_\lambda\coloneqq Z_{\lambda,1}, \quad 
B_\lambda \coloneqq Z_{\lambda,2}, \quad
C_\lambda \coloneqq \overline{Z_{\bar{\lambda},1}}.
$$
Note that $C_\lambda$ is well defined since $\bar{\lambda}\in \Lambda$.
Moreover, denote $A=(A_\lambda)_{\lambda\in\Lambda_+}$, $B=(B_\lambda)_{\lambda \in\Lambda_+}$, and $C=(C_\lambda)_{\lambda\in\Lambda_+}$.
It is obvious, that $A,B$ are $f$-close to $\Lambda_+$. 
Since $Z_{\bar{\lambda},1}\in B_{f(\bar{\lambda})}(\bar{\lambda})$ we have for all $\lambda\in \Lambda_+$ that 
$$
|C_\lambda-\lambda| = \big| \overline{Z_{\bar{\lambda},1}}-\lambda\big|  
= 
\big| Z_{\bar{\lambda},1}-\bar{\lambda}\big| \le f(\bar{\lambda}) = f(\lambda).
$$
This shows that $C$ is $f$-close to $\Lambda_+$.
Notice that $\mathcal{U}=\{Z_{\lambda,\ell},\, (\lambda,\ell)\in J\}$ satisfies 
$$
\mathcal{U} \cup \overline{\mathcal{U}} \supseteq A \cup B \cup C.
$$
We fix $\beta \in (0,\gamma-2\alpha)$ arbitrary. According to Proposition \ref{prop:1214}(1) (with $\alpha=\pi$ and $\Lambda=\Lambda_+$), it suffices to show that the sequence 
$$
\eta_\lambda \coloneqq \frac{|\lambda| e^{-\beta|\lambda|^2}}{\varphi(A_\lambda,B_\lambda,C_\lambda)},\quad \lambda\in\Lambda_+,
$$
is almost surely bounded from above. Consider the events
$
E\coloneqq \{ \sup_{\lambda\in\Lambda_+} \eta_\lambda <\infty\},
$
and
$$
E_L \coloneqq \left \{ \forall \lambda \in \Lambda_+ : \eta_\lambda \leq L \right \}, \quad L>0.
$$
Since $E\supseteq E_L$ for all $L>0$, it suffices to show that $\mathbb{P}[E_L]\to 1$ as $L\to \infty$ in order to conclude that $\mathbb{P}[E]=1$.
We proceed similarly as in the proof of Theorem \ref{thm:Main4} and estimate
\begin{align*}
    \mathbb{P}[E_L] &= 1 - \mathbb{P}[\exists \lambda\in\Lambda_+:\, \eta_\lambda > L]\\
    & \ge 1 - \sum_{\lambda\in\Lambda_+} \mathbb{P}\left [\varphi(A_\lambda,B_\lambda,C_\lambda) < \frac1{L} |\lambda| e^{-\beta|\lambda|^2} \right ]. 
\end{align*}
We claim that for all $\lambda\in\Lambda_+$ and $\varepsilon>0$, it holds that 
$$
\mathbb{P}[\varphi(A_\lambda,B_\lambda,C_\lambda) < \varepsilon] \le 4 \varepsilon.
$$
Indeed, if $\lambda \in \Lambda_+\cap \R$, the statement follows from Lemma \ref{lem:2randompts}.
On the other hand, if $\lambda\in \Lambda_+ \cap \mathbb{H}_+$ we have that 
$$
(A_\lambda, B_\lambda, C_\lambda) = (Z_{\lambda,1}, Z_{\lambda,2}, \overline{Z_{\bar{\lambda},1}})
$$
are independent random complex variables, with each of them uniformly distributed on the disk $B_{f(\lambda)}(\lambda)$.
In this case, we can resort to Lemma \ref{lem:3randompts}.
Consequently, we have that
$$
\mathbb{P}[E_L] \ge 1 - \frac{4}{L}\sum_{\lambda\in\Lambda_+} |\lambda| e^{-\beta|\lambda|^2} \to 1, \quad L\to \infty,
$$
and are done.
\end{proof}

In the setting of Theorem \ref{thm:Main6}, we produce a uniqueness set of density $8+\varepsilon$, with $\varepsilon>0$ arbitrarily small.
In fact, the proof shows that 
$$
\{ Z_{\lambda,1},\, \lambda\in \Lambda\} \cup \{ Z_{\lambda,2},\, \lambda\in \Lambda_+\}  
$$
forms a uniqueness set almost surely.
A quarter of the available information, i.e., samples at
$$
Z_{\lambda,2}, \quad \lambda\in \Lambda\setminus \Lambda_+,
$$
is actually not used at all. This partly explains why we are able to further push down the lower bound on the density to $>6$, cf. Theorem \ref{thm:density_lower_bound}. A related statement holds with regards to the result where we restrict to signals in $L^2_e(\R,\R)$. The proof of the corresponding theorem comes next.
 
\begin{proof}[Proof of Theorem \ref{thm:Main7}]
    Identifying $\R^2$ with the complex plane, it suffices to show that $\mathcal{U}$ is a uniqueness set for the phase retrieval problem in $\mathcal{B}(L^2_e(\R,\R))=\ft_\pi(\C)\cap \mathcal{S}_-^*$ almost surely.
    
    To do so, we denote by $\Lambda_0$ the set of lattice points which are in the closed first quadrant minus the origin, that is, 
    $$
    \Lambda_0\coloneqq \{\lambda\in\Lambda\setminus \{0\} : \re{(\lambda)} \geq 0, \, \im{(\lambda)}\ge 0\}.
    $$
    Since $\Lambda \in \mathcal{L}$, it follows that
    $$
    \Lambda_0 \cup \overline{\Lambda_0} \cup (-\Lambda_0) \cup (- \overline{\Lambda_0}) = \Lambda \setminus \{0\},
    $$
    which implies that $\Lambda_0$ is a $\nicefrac14$-Liouville set for $\ft_{4\pi}(\C)$.
    For each $\lambda\in\Lambda_0$, we define three complex random variables via 
    $$
    A_\lambda\coloneqq Z_\lambda,\quad B_\lambda\coloneqq \overline{Z_{\bar{\lambda}}},\quad C_\lambda\coloneqq -Z_{-\lambda}.
    $$
    Clearly, $A=(A_\lambda)_{\lambda\in\Lambda_0}$ is $f$-close to $\Lambda_0$.
    It is not difficult to see, that the same holds true for $B=(B_\lambda)_{\lambda\in\Lambda_0}$ and $C=(C_\lambda)_{\lambda\in\Lambda_0}$.
    Indeed, if $\lambda\in\Lambda_0$ we have that 
    $$
    |B_\lambda-\lambda| = |\overline{Z_{\bar{\lambda}}} - \lambda| = |Z_{\bar{\lambda}}-\bar{\lambda}| \le f(\bar{\lambda)} = f(\lambda).
    $$
    Similarly, 
    $$
    |C_\lambda-\lambda| = |- Z_{-\lambda}-\lambda| = |Z_{-\lambda} - (-\lambda)| \le f(-\lambda) =f(\lambda).
    $$
    Moreover, we have that 
    $$
    \mathcal{U}\cup \overline{\mathcal{U}} \cup (-\mathcal{U})  \cup (-\overline{\mathcal{U}}) \supseteq A\cup B \cup C.
    $$
    Let $\beta\in (0,\gamma-2\alpha)$ be fixed and define a sequence of random variables by 
    $$
    \eta_\lambda \coloneqq \frac{|\lambda|e^{-\beta|\lambda|^2}}{\varphi(A_\lambda,B_\lambda,C_\lambda)}, \quad \lambda\in\Lambda_0.
    $$
    Proposition \ref{prop:1214}(2) states that $\mathcal{U}$ is a set of uniqueness for phase retrieval in $\ft_{4\pi}(\C)$ provided that $(\eta_\lambda)_{\lambda\in \Lambda_0}$ is bounded.
    Thus, it suffices to show that 
    $$
    E\coloneqq \left \{\sup_{\lambda\in\Lambda_0} \eta_\lambda < \infty \right \}
    $$
    occurs with probability one.
    We introduce 
    $$
    E_L\coloneqq \{ \forall \lambda\in\Lambda_0:\, \eta_\lambda\le L\},\quad L>0.
    $$
    Since $E\supseteq E_L$ for all $L>0$, it suffices to show that $\mathbb{P}[E_L]\to 1$ as $L\to \infty$ in order to conclude $\mathbb{P}[E]=1$. 
    Note that
    \begin{align*}
    \mathbb{P}[E_L] &= 1 - \mathbb{P}[\exists \lambda\in\Lambda_0 :  \eta_\lambda > L]\\
    & \ge 1 - \sum_{\lambda\in\Lambda_0} \mathbb{P}\left [\varphi(A_\lambda,B_\lambda,C_\lambda) < \frac1{L} |\lambda| e^{-\beta|\lambda|^2} \right ] 
\end{align*}
Given $\lambda\in\Lambda_0$, then  a) $\lambda$ is in the open first quadrant, b) $\lambda$ is on the real axis or c) $\lambda$ is on the imaginary axis. In case a) we have that 
$$
(A_\lambda,B_\lambda,C_\lambda) = (Z_\lambda, \overline{Z_{\bar{\lambda}}}, - Z_{-\lambda})
$$
is a triple of i.i.d. complex random variables, uniformly distributed on $B_{f(\lambda)}(\lambda)$.
Thus, it follows from Lemma \ref{lem:3randompts} that 
\begin{equation}\label{eq:boundprobvarphi}
    \mathbb{P}[\varphi(A_\lambda,B_\lambda,C_\lambda)< \varepsilon] \le 4 \varepsilon, \quad \varepsilon > 0.
\end{equation}
In case b) we have that 
$$
(A_\lambda,B_\lambda,C_\lambda) = (Z_\lambda, \overline{Z_\lambda}, - Z_{-\lambda}).
$$
Hence, $B_\lambda=\overline{A_\lambda}$. Note that $C_\lambda$ is uniformly distributed on $B_{f(\lambda)}(\lambda)$ and independent from $(A_\lambda,B_\lambda)$.
Thus, it follows from Lemma \ref{lem:2randompts} that \eqref{eq:boundprobvarphi} also holds true in case b).
For case c) one argues in a similar way. As a result we have that 
$$
\mathbb{P}[E_L] \ge 1 - \frac4{L} \sum_{\lambda\in\Lambda_0} |\lambda| e^{-\beta|\lambda|^2} \to 1, \quad L \to \infty.
$$
\end{proof}

\subsection{Gabor phase retrieval, density reduction, and separateness}\label{subsec:gabordetrestriction}

In this section we prove Theorem \ref{thm:density_lower_bound}. Observe, that Theorem \ref{thm:density_lower_bound}(1) follows directly from Theorem \ref{thm:Main1} and the fact that every lattice $\Lambda \subseteq \R^2$ of density $D(\Lambda) > 4$ is a Liouville set for $\ft_{4\pi}(\C)$ (see Theorem \ref{thm:Main3}). It remains to show the second and third claim of Theorem \ref{thm:density_lower_bound}, which state that for every $d>6$ resp. $d>3$, there exists a uniformly distributed uniqueness set for the Gabor phase retrieval problem in $L^2(\R,\R)$ resp. $L^2_e(\R,\R)$ having density $d$. In the latter case, the uniqueness set can be chosen to be separated. Before turning to the proofs of the statements, we require an elementary lemma related to translates of lattice Liouville sets.

\begin{lemma}\label{lma:liouville_shift}
    Let $\Lambda \subseteq \C$ be a lattice, and let $\alpha >0$. If $\Lambda$ is a Liouville set for $\ft_\alpha(\C)$, then $\Lambda-w$ is a Liouville set for $\ft_\alpha(\C)$ for every $w \in \C$.
\end{lemma}
\begin{proof}
    The lattice $\Lambda$ is a Liouville set for $\ft_\alpha(\C)$, if and only if $s(\Lambda) < \frac{\pi}{\alpha}$. Let $\varepsilon>0$ such that $s(\Lambda) < \frac{\pi}{\alpha + \varepsilon}$. Suppose that $F \in \ft_\alpha(\C)$ is bounded on $\Lambda - w$ for some $w \in \C$. According to equation \eqref{eq:shift_episolon}, $F(\cdot - w)$ is a function belonging to $\ft_{\alpha + \varepsilon}(\C)$, and this function is bounded on $\Lambda$. Theorem \ref{thm:Main3} yields the assertion.
\end{proof}

In the following proof we make use of the fact that if $\Lambda$ has uniform density $D(\Lambda) = d$, and $A$ is uniformly close to $\Lambda$, then $D(A)=d$. This follows, for instance, from \cite[Lemma 4.23]{zhu:fock}. Moreover, the definition of the uniform density directly implies that if $A$ and $B$ are disjoint and have uniform density $d_A$ and $d_B$, respectively, then $D(A \cup B) = d_A+d_B$.

\begin{proof}[Proof of Theorem \ref{thm:density_lower_bound}(2)]

Let $\Gamma = v(\Z + i \Z)$ be a square lattice with $0 < v < \frac{1}{2}$. Then $D(\Gamma) = \frac{1}{v^2} > 4$, and therefore $\Gamma$ is a Liouville set for $\ft_{4\pi}(\C)$. Further, consider the shifted lattice
$$
\Lambda \coloneqq \Gamma + \frac{v}{2}i.
$$
In view of Lemma \ref{lma:liouville_shift}, the set $\Lambda$ is a Liouville set for $\ft_{4\pi}(\C)$, and is symmetric with respect to both coordinate axes.
Let $\Gamma' \subseteq \Gamma$ be the sub-lattice defined by
$$
\Gamma' \coloneqq v\{ m(1+i) + n(1-i) : (m,n) \in \Z^2 \}.
$$
Then $D(\Gamma') = \frac{1}{2} D(\Gamma) = \frac{1}{2v^2}$. Defining
$$
\Lambda' \coloneqq \Gamma' + \frac{v}{2}i,
$$
it follows that
$$
\Lambda' \sqcup \overline{\Lambda'} = \Lambda,
$$
where the symbol $\sqcup$ indicates that the union is disjoint.
In particular, this shows that $\Lambda'$ is a $\nicefrac12$-Liouville set for $\ft_{4\pi}(\C)$.
We proceed by choosing sequences $A=(a_\lambda)_{\lambda\in\Lambda'},B=(b_\lambda)_{\lambda\in\Lambda'},C=(c_\lambda)_{\lambda\in\Lambda'} \subseteq \C$, that are uniformly noncollinear and $f$-close to $\Lambda'$.
Since $\Lambda'$ is a $\nicefrac12$-Liouville set for $\ft_{4\pi}(\C)$, it follows from Proposition \ref{prop:1214}(1), that $\mathcal{U} \coloneqq A\cup B\cup C$ is a uniqueness set for the phase retrieval problem in $\ft_\pi(\C)\cap \mathcal{S}^*$.
Thus, $\mathcal{U}$ is a uniqueness set for the Gabor phase retrieval problem in $L^2(\R,\R)$.\footnote{Notice, that a specific choice for the sequence $A$ is $A=\Lambda'$, and an alternative to the uniqueness sets $\mathcal{U} \coloneqq A\cup B\cup C = \Lambda' \cup B \cup C$ is given by the union $\mathcal{U}_2 = \Lambda' \cup \overline{B} \cup C$. Figure \ref{fig:3} visualizes the set $\mathcal{U}_2$ for specific choices of the $f$-close sets $B$ and $C$.} 

Since $\Gamma'$ is a sub-lattice with density $D(\Gamma')=\frac{1}{2v^2}$, its shifted version $\Lambda'$ also has density $D(\Lambda')=\frac{1}{2v^2}$. Moreover, as $A,B,C$ are pairwise disjoint and $f$-close to $\Lambda'$, we get that
$$
D(\mathcal{U}) = \frac3{2v^2}.
$$
Since $v$ can be as close to $\frac{1}{2}$ as we please, the statement follows.
\end{proof}

For the space $L^2_e(\R,\R)$, a further reduction of the density by a factor $\frac{1}{2}$ compared to the space $L^2(\R,\R)$ can be performed. In contrast to all previous results, also separateness occurs.

\begin{proof}[Proof of Theorem \ref{thm:density_lower_bound}(3)]
Let $\Gamma = v(\Z + i \Z)$ be a square lattice with $v < \frac{1}{2}$. Then $D(\Gamma) = \frac{1}{v^2} > 4$, and therefore $\Gamma$ is a Liouville set for $\ft_{4\pi}(\C)$. Consider the shifted lattice
$$
\Lambda \coloneqq \Gamma + \frac{v}{2}(1+i),
$$
enumerated in the following natural order:
$$
\Lambda = \left \{  \lambda_{mn} = vm + \frac{v}{2} + i \left (vn + \frac{v}{2} \right ) : (m,n) \in \Z^2 \right \}.
$$
In view of Lemma \ref{lma:liouville_shift}, $\Lambda$ is a Liouville set for $\ft_{4\pi}(\C)$, and is symmetric with respect to both coordinate axes.

We extract a $\nicefrac14$-Liouville set out of $\Lambda$ as follows: let $Q_1,Q_2,Q_3,Q_4 \subseteq \Lambda$ denote columns of the shifted lattice $\Lambda$ belonging to the four quadrants (i.e., $Q_1$ belongs to the first quadrant, $Q_2$ belongs to the second quadrant, ...), defined by
\begin{equation}
    \begin{split}
        Q_1 &\coloneqq \{ \lambda_{4m,n} : m,n \geq 0 \}, \\
        Q_2 &\coloneqq \{ \lambda_{-2-4m,n} : m,n \geq 0 \}, \\
        Q_3 &\coloneqq \{ \lambda_{-3-4m,-n-1} : m,n \geq 0 \}, \\
        Q_4 &\coloneqq \{ \lambda_{3+4m,-n-1} : m,n \geq 0 \}.
    \end{split}
\end{equation}
Further, let $\Lambda' \coloneqq Q_1 \sqcup Q_2 \sqcup Q_3 \sqcup Q_4$.
It follows from the definition of $Q_1,Q_2,Q_3,Q_4$, that
$$
\Lambda = \Lambda' \cup \overline{\Lambda'} \cup (-\Lambda') \cup (-\overline{\Lambda'}).
$$
Hence, $\Lambda'$ is a $\nicefrac14$-Liouville set for $\ft_{4\alpha}(\C)$. Further, $\Lambda'$ is uniformly distributed with density $D(\Lambda') = \frac{1}{4} D(\Lambda) = \frac{1}{4v^2}$. We choose sequences $A=(a_\lambda)_{\lambda\in\Lambda'},B=(b_\lambda)_{\lambda\in\Lambda'},C=(c_\lambda)_{\lambda\in\Lambda'} \subseteq \C$ that are uniformly noncollinear and $f$-close to $\Lambda'$. Further, we set
$$
\mathcal{U} \coloneqq \overline{A} \cup (-B) \cup (-\overline{C}).
$$
Proposition \ref{prop:1214}(2) implies that $\mathcal{U}$ is a uniqueness set for the phase retrieval problem in $\ft_\pi(\C) \cap \mathcal{S}^*_-$. Thus, it is a uniqueness set for the Gabor phase retrieval problem in $L^2_e(\R,\R)$.\footnote{The uniqueness set $\mathcal{U} = \overline{A} \cup (-B) \cup (-\overline{C})$ is visualized in Figure \ref{fig:4}.}

Now observe, that $\mathcal{U}$ admits an enumeration by means of the elements of $\Lambda \setminus \Lambda'$, such that $\mathcal{U}$ is $f$-close to $\Lambda \setminus \Lambda'$. This implies that $\mathcal{U}$ is separated. Further, $\Lambda \setminus \Lambda'$ is uniformly distributed with density $D(\Lambda \setminus \Lambda')=\frac{3}{4v^2}$. Using $f$-closeness once more, we conclude that $D(\mathcal{U}) = D(\Lambda \setminus \Lambda')=\frac{3}{4v^2}$, where $v$ can be as close to $\frac{1}{2}$ as we please. This yields the assertion.
\end{proof}

\begin{figure}
\centering
\hspace*{-0.4cm}
  \includegraphics[width=12.7cm]{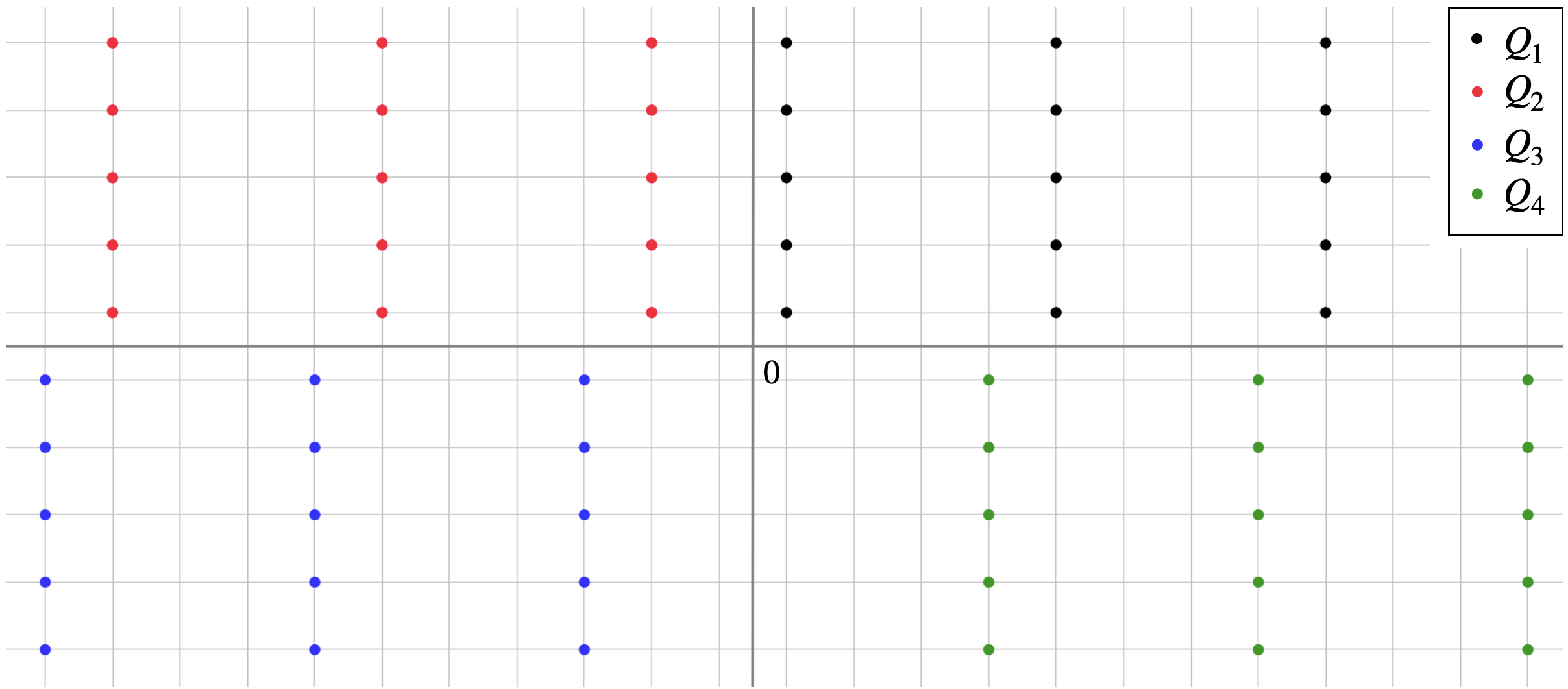}
\caption{Visualisation of the sets $Q_1,Q_2,Q_3$, and $Q_4$ used in the proof of Theorem \ref{thm:density_lower_bound}(3). The union $\Lambda' \coloneqq Q_1 \sqcup Q_2 \sqcup Q_3 \sqcup Q_4$ is a $\nicefrac14$-Liouville set for $\ft_{4\alpha}(\C)$ (the intersections of the grey mesh are the points of the square-lattice $\Lambda$).
}
\label{fig:proof}
\end{figure}

\section*{Acknowledgements}
Martin Rathmair was supported by the Erwin–Schr{\"o}dinger Program (J-4523) of
the Austrian Science Fund (FWF).

\bibliographystyle{acm}
\bibliography{bibfile}

\end{document}